\documentclass{aptpub}

\authornames{E.~Foxall} 
\shorttitle{Logistic process} 

\usepackage{amsmath,amssymb}
\usepackage[margin=1in]{geometry}
\usepackage{comment,enumitem}
\usepackage{xcolor}

\def\bal{\begin{align}}
\def\eal{\end{align}}
\def\cd{ \ \stackrel{(d)}{\rightarrow} \ }
\def\cp{ \ \stackrel{(p)}{\rightarrow} \ }
\def\dlt{\delta}
\def\sgn{\mathrm{sgn}}

\def\tld{\tilde}
\def\lf{\lfloor}
\def\rf{\rfloor}
\def\gma{\gamma}
\def\nid{\noindent}
\def\N{\mathbb{N}}

\def\1{\mathbf{1}}
\def\eps{\epsilon}

\def\ol{\overline}

\def\nid{\noindent}
\def\tbf{\textbf}
\def\enumrom{\begin{enumerate}[noitemsep,label={(\roman*)}]}
\def\enumar{\begin{enumerate}[noitemsep,label={\arabic*.}]}
\def\enumalph{\begin{enumerate}[noitemsep,label={\alph*)}]}
\def\enumend{\end{enumerate}}
\def\itemgo{\begin{itemize}[noitemsep]}
\def\itemend{\end{itemize}}
\def\P{\mathbb{P}}
\def\E{\mathbb{E}}
\def\Q{\mathbb{Q}}

\newcommand{\R}{\mathbb{R}}
\newcommand{\Z}{\mathbb{Z}}
\newcommand{\ep}{\epsilon}

\begin{document}

\title{Extinction time of the logistic process} 

\authorone[University of British Columbia, Okanagan Campus]{Eric Foxall} 

\addressone{3333 University Way, Kelowna BC Canada V1V 1V7} 
\emailone{efoxall@mail.ubc.ca} 

\begin{abstract}
The logistic birth and death process is perhaps the simplest stochastic population model that has both density-dependent reproduction, and a phase transition, and a lot can be learned about the process by studying its extinction time, $\tau_n$, as a function of system size $n$. A number of existing results describe the scaling of $\tau_n$ as $n\to\infty$, for various choices of reproductive rate $r_n$ and initial population $X_n(0)$ as a function of $n$. We collect and complete this picture, obtaining a complete classification of all sequences $(r_n)$ and $(X_n(0))$ for which there exist rescaling parameters $(s_n)$ and $(t_n)$ such that $(\tau_n-t_n)/s_n$ converges in distribution as $n\to\infty$, and identifying the limits in each case.
\end{abstract}



\section{Introduction}
Fix $n\in \N$ and $r\in \R_+$ and consider the continuous-time Markov chain with
\begin{align}\label{eq:logis-model}
X \to \begin{cases}X+1 & \text{at rate} \quad rX(1-X/n) \\
         X-1 & \text{at rate} \quad X.
        \end{cases}
\end{align}
This birth and death process has a simple interpretation in terms of infection spread: in a population of $n$ individuals, $X$ of them are type $I$ (infectious), and the remaining $n-X$ individuals are type $S$ (susceptible). Each type $I$ individual, at rate $r$, selects an individual uniformly at random from the population and infects them. In addition, each type $I$ individual recovers at rate $1$ and is once again immediately susceptible. Since the state transitions at the individual level are $S\to I \to S$, this is called the SIS model.\\

This model first appears, in a somewhat more general form, in a probabilitistic treatment of logistic population growth by Feller, in Section 4 of \cite{feller-1939german} (see \cite{feller-works1}, pages 471--495, for an english translation), where it is shown that the expected value of the process at each point in time is bounded above by the solution to the corresponding logistic differential equation, namely, $X'=rX(1-X/n)-X$. Later, Kurtz \cite{Kurtz-DDMC} gives limit theorems that apply to the sequence of processes $x_n=X_n/n$ (using the subscript to emphasize dependence on $n$), when $r$ is fixed and $\lim_{n\to\infty}x_n(0)$ exists, demonstrating convergence, on compact time intervals, to the solution of the logistic ODE $x'=rx(1-x)-x$ with $x(0)=\lim_{n\to\infty}x_n(0)$, with Gaussian fluctuations of order $1/\sqrt{n}$ that solve an explicit SDE. Extensions of the ODE convergence to longer (growing slowly with $n$) time intervals are possible under certain conditions, see for example the work of Barbour, Chigansky and Klebaner \cite{rand-init}.\\

Later work has focused on understanding the long-term behaviour of the process for large $n$, in particular (i) the time to extinction $\tau_n:=\inf\{t\colon X_n(t)=0\}$ and (ii) the behaviour prior to extinction. In this article we shall focus on the distribution of $\tau_n$, as $n\to\infty$, as a function of the reproductive parameter $r_n$ and the initial value $X_n(0)$, when both are allowed to vary with $n$. We do so by identifying \emph{distributional limits} of $(\tau_n)$, i.e., deterministic sequences $(s_n)$, $(t_n)$ and random variables $T$ such that $(\tau_n-t_n)/s_n$ converges in distribution to $T$ as $n\to\infty$. Much of the work on this problem already exists; here, we provide a survey and concise proofs of known results, then fill in the remaining gaps to produce a complete theory on the topic.\\

There is a parallel line of inquiry on model \eqref{eq:logis-model} that concerns the so-called quasi-stationary distribution (QSD), which is the unique stationary distribution conditioned on non-extinction; see the comparatively recent monograph of Nasell \cite{nasellmono} for a comprehensive survey on this topic, which also includes a discussion of the three phases of the model that we explain below and that are key to understanding the behaviour of the model. As explained in \cite{nasellmono}, if $X_n(0)$ has the QSD then $\tau_n$ is exponentially distributed for each $n$, and by studying the QSD we can obtain estimates of $\E[\tau_n]$, when $X_n(0)$ has either the QSD or the deterministic value $1$. Generally speaking, $\E[\tau_n]$ corresponds to $(t_n)$ in the rescaling described above, then further methods are needed to find $(s_n)$ and $T$. These methods appear progressively over time in works that are not surveyed in \cite{nasellmono}. So, by gathering these methods and previous works and completing the picture, we obtain a complementary perspective to the path taken in \cite{nasellmono}.\\

As is easily shown, the ODE $x'=rx(1-x)-x$ undergoes a transcritical bifurcation at $(x,r)=(0,1)$, and for $r>1$ has the endemic equilibrium $x_\star=1-1/r$. An important question is whether this is mirrored by the behaviour of the stochastic process. Kryscio and Lef{\`e}vre \cite{kryscio} estimate $\E[\tau_n]$, as well as the QSD for large $n$; in particular, they find that for $r<1$, $\E[\tau_n]=O(\log n)$ uniformly over $X_n(0)$ and that the QSD is concentrated near $0$, while for $r>1$, $\E[\tau_n]$ grows exponentially with $n$ uniformly over $X_n(0)$ and the QSD of $x_n$ is concentrated near $x_\star$. Later, Andersson and Djehiche \cite{thresh-lim} refine results on extinction time to limit theorems for the distribution of $\tau_n$ when $r<1$ and $r>1$, in the following cases: (i) $X_n(0)$ is a constant independent of $n$, and (ii) $\lim_{n\to\infty}x_n(0)$ exists and is positive. When $r<1$ and $X_n(0)$ is constant, for large $n$, $X_n(t)$ is well approximated by a branching process up until extinction, so $\tau_n$ converges in distribution to the extinction time of the branching process. When $r<1$ and $\lim_{n\to\infty}x_n(0)>0$, $(1-r)\tau_n-t_n$ converges in distribution to a standard Gumbel $G$, with $\P(G\le w)=e^{-e^{-w}}$, for some deterministic correction $t_n$ of order roughly $\log n$. As pointed out by Doering, Sargsyan and Sander, \cite{doering2005}, the formula given in \cite{thresh-lim} has an error in this case; the corrected formula in the case $r<1$ can be found for ex.~in the recent work of Brightwell, House and Luczak \cite{bright-luz}. When $r>1$ and $\lim_{n\to\infty} x_n(0)>0$, $\tau_n/\E[\tau_n]$ converges in distribution to exponential with rate $1$. The basic reasoning is that $x_n(t)$ is metastable around $x_\star$, with Ornstein-Uhlenbeck-type fluctuations of order $1/\sqrt{n}$, which can be deduced from the results of Kurtz \cite{Kurtz-DDMC}, so a rare and rather sudden event is needed to cause extinction.\\

The behaviour near $r=1$ is more delicate. Nasell \cite{nasell1996} identifies the transition region $r_n - 1=O(1/\sqrt{n})$, also known as transition window or critical regime, and finds that for $r_n=1+c/\sqrt{n}$, $\E[\tau_n] \sim f(c)\sqrt{n}$ for some function $f$, when the distribution of $X_n(0)$ is the QSD. For the same scale of $r_n$, Dolgoarshinnykh and Lalley \cite{crit-scale} prove convergence in distribution of $Y_n(t):=X_n(\sqrt{n} t)/\sqrt{n}$ to the solution of the modified Feller diffusion appearing below in Theorem \ref{thm:nonlin-diff}, when $\lim_{n\to\infty} Y_n(0)$ exists. Brightwell, House and Luczak \cite{bright-luz} obtain the distribution of $\tau_n$ throughout the subcritical regime, i.e., when $\lim_{n\to\infty} r_n \le 1$ and $\sqrt{n}(1-r_n)\to\infty$, for initial values satisfying $X_n(0)(1-r_n)\to\infty$, with particular focus on the barely subcritical case $1-r_n=o(1)$. Until now, the barely supercritical case, where $r_n-1=o(1)$ and $\sqrt{n}\,(r_n-1)\to\infty$, has remained unsolved. \\

In this article we collect and complete existing results concerning the distribution of $\tau_n$ for large $n$ in order to obtain a complete understanding of the extinction time, for all possible choices of $r_n$ and $X_n(0)$. Since the existing results form a somewhat overlapping patchwork of the different cases, I have opted to include a full and self-contained proof of all the different cases, which allows us to more accurately view the true extent of the different methods used. Generally speaking, we shall follow existing methods, although in some cases some improvement was possible; I shall point out when each is the case. In addition, the barely supercritical case is new, and in fact it requires the most effort.\\

The breadth of our results can be summarized as follows. Let $\dlt_n = r_n-1$, $a_n = (|\dlt_n| + 1/\sqrt{n})\,X_n(0)$ and $c_n = \sqrt{n}\,\dlt_n$. Suppose that $X_n(0)\to X_\infty(0) \in \N\cup \{\infty\}$, $r_n\to r_\infty \in [0,\infty)$, $a_n\to a_\infty \in [0,\infty]$ and $c_n\to c_\infty \in [-\infty,\infty]$ as $n\to\infty$. Then there exist sequences $(s_n),(t_n)$ and a non-degenerate distribution function $F:[0,\infty) \to [0,1]$ (that we identify in every case) such that  $\P((\tau_n-t_n)/s_n \le w) \to F(w)$ as $n\to\infty$ at continuity points $w$ of $F$. Moreover, $F$ depends only on the values of the limits $X_\infty(0),r_\infty,a_\infty$ and $c_\infty$. As explained below, $a_n$ measures whether initial values are small or large, and $c_n$ determines the phase of the process.

\section{Phase diagram and main results}

Let us formally state the setting and the assumptions that shall hold throughout the paper. Given $n \in \N$, $r_n\in[0,\infty)$ and $X_n(0) \in \{0,\dots,n\}$, $X_n$ denotes the Markov chain $(X_n(t))_{t \ge 0}$ with initial value $X_n(0)$ and transition rates

\begin{align}\label{eq:logis-model}
X_n \to \begin{cases}X_n+1 & \text{at rate} \quad r_nX_n(1-X_n/n) \\
         X_n-1 & \text{at rate} \quad X_n.
        \end{cases}
\end{align}
The extinction time of $X_n$ is $\tau_n=\inf\{t\colon X_n(t)=0\}$. We fix a sequence of values $(X_n(0))_{n \ge 1}$ and $(r_n)_{n\ge 1}$ and seek deterministic sequences $(t_n)$ and $(s_n)$ and a random variable $T$ such that $(\tau_n-t_n)/s_n$ converges in distribution to $T$ as $n\to\infty$, that we refer to as a \emph{distributional limit} for $(\tau_n)$, subject to the following assumptions.\\

\nid\tbf{Assumptions.} Let $\delta_n=r_n-1$, $a_n=(\,|\delta_n| + 1/\sqrt{n})\, X_n(0)$ and $c_n = \sqrt{n}\,\delta_n$. As $n\to\infty$, it is assumed that
\begin{align*}
X_n(0) \to X_\infty(0) \in \N\cup\{\infty\}, \quad r_n \to r_\infty \in [0,\infty), \\
a_n \to a_\infty \in [0,\infty] \quad \text{and} \quad c_n \to c_\infty \in [-\infty,\infty].
\end{align*}
\\
The behaviour of $\tau_n$ depends on $r_n$ and $X_n(0)$, or equivalently $r_n$ and $x_n(0)=X_n(0)/n$. So, we can organize our results into a phase diagram for $(r,x) \in \R_+ \times [0,1]$. As observed by Nasell \cite{nasell1996}, parameter values are effectively partitioned into three phases:
\begin{enumerate}[noitemsep]
 \item \textbf{Subcritical}: $c_n \to -\infty$ as $n\to\infty$,
 \item \textbf{Critical}: $c_n \to c_\infty \in \R$ as $n\to\infty$,
 \item \textbf{Supercritical}: $c_n \to \infty$ as $n\to\infty$.
\end{enumerate}
It's important to note that since $c_n=\sqrt{n}\,(r_n-1)$, the critical phase has width $O(1/\sqrt{n})=o(1)$ around the point $r=1$. Along with parameter values, there is a dependence on the initial value $X_n(0)$. The following more refined partition into six regions delineates the different regimes for the limit behaviour of $\tau_n$, as follows:
\enumar
\item \nid\tbf{Discrete}: $X_\infty(0) \in \N$ and $r_\infty \le 1$.

\item \nid\tbf{Linear diffusive}: $X_\infty(0)=\infty$ and either
\begin{enumerate}[noitemsep]
\item subcritical phase and $a_\infty<\infty$, or
\item critical/supercritical phase and $a_\infty= 0$.
\end{enumerate}

\item \nid\tbf{Subthreshold cutoff}: subcritical phase and $a_\infty = \infty$.

\item \nid\tbf{Non-linear diffusive}: critical phase and $a_\infty>0$.

\item \nid\tbf{Threshold}: supercritical phase and $0<a_\infty<\infty$.

\item \nid\tbf{Metastable}: supercritical phase and $a_\infty=\infty$.
\enumend

\nid In each region, the following method is used to establish the distributional limit for $(\tau_n)$.

\enumar
\item \nid\tbf{Discrete}: as observed by Andersson and Djehiche \cite{thresh-lim}, in this region, for each $n$, $X_n$ can be coupled to the linear birth-and-death process $Z_n$ with $Z_n(0)=X_n(0)$ and transition rates
\begin{align}\label{eq:lbdp}
Z_n \to \begin{cases}Z_n+1 & \text{at rate} \quad r_n\, Z_n, \\
         Z_n-1 & \text{at rate} \quad Z_n
        \end{cases}
\end{align}
in such a way that $\P(X_n(t)=Z_n(t) \ \  \text{for all} \ \  t \in [0,\tau_n] \, ) = 1-o(1)$. \\

\item \nid\tbf{Linear diffusive}: as observed by Dolgoarshinnykh and Lalley \cite{crit-scale}, in this region, $X_n(X_n(0)\, t)/X_n(0)$ converges in distribution to the diffusion $Y$ with $Y(0)=1$ and
\begin{align}\label{eq:lin-diff}
dY = -a_\infty \, Y dt + \sqrt{2Y}dB,
\end{align}
where $B$ is standard Brownian motion, and it is not hard to strengthen this result to show that $\tau_n/X_n(0)$ converges in distribution to $\inf\{t \colon Y(t)=0\}$. This explains nicely the form of the limit. The method we will use instead, since it applies across several regions, is that of Brightwell, House and Luczak \cite{bright-luz}, which is essentially the same coupling method as \cite{thresh-lim} but extends their approach by not requiring that $X_n(t)=Z_n(t)$ for all $t\le \tau_n$ with high probability, but instead using both upper and lower bounds $(Z_n)$ and $(Z_n')$, both linear birth-and-death processes, such that the extinction time of both is comparable to $(\tau_n)$. To define the lower bound $(Z_n')$, a ceiling $M_n$ is defined such that the coupling is valid so long as $X_n \le M_n$ and such that $X_n$ hits $0$ before $M_n$ with high probability.\\

\item \nid\tbf{Subthreshold cutoff}: as observed by Brightwell, House and Luczak \cite{bright-luz}, $x_n$ is well approximated by the logistic ODE
$x'=r_\infty\, x\,(1-x)$ until $|\dlt_n|\, X_n =O(1)$ at which point $X_n$ enters the linear diffusive regime and the above coupling method can be used. In practice, they show the coupling method works can be applied once $|\dlt_n|\, X_n \le \omega_n$ for some $(\omega_n)$ tending to $\infty$ slowly enough as $n\to\infty$ as a function of $(c_n)$.\\

\item \nid\tbf{Non-linear diffusive}: as observed by Dolgoarshinnykh and Lalley \cite{crit-scale}, the process $X_n(\sqrt{n}\, t)/\sqrt{n}$ converges in distribution to the diffusion $Y$ with
$Y(0)=\lim_{n\to\infty} X_n(0)/\sqrt{n}$ and 
$$dY = (c_\infty Y - Y^2)\,dt + \sqrt{2Y}dB.$$
In this case we make use of this diffusion limit combined with a result that says $\limsup_n \P(\tau_n/\sqrt{n} >t \mid X_n(0) \le \alpha \sqrt{n}) \to 0$ for each $t>0$ as $\alpha>0$ to establish convergence in distribution of $\tau_n/\sqrt{n}$ to $\inf\{t\colon Y(t)=0\}$. \\

\item \nid\tbf{Threshold}: let $x_n^\star=1-1/r_n$, $X_n^\star = \lf n x_n^\star \rf$ and $\tau_n^\star=\inf\{t\colon X_n(t) \in \{0,X_n^\star\}\}$. Then as observed by Andersson and Djehiche \cite{thresh-lim} when $r_n=r_\infty>1$ and $X_n(0)=X_\infty(0)$ for all $n$, and extended in this article to the whole supercritical phase,
$$\lim_{n\to\infty}\P(X_n(\tau_n^\star) =0 )=\begin{cases} r_\infty^{-X_\infty(0)} & \text{if} \ X_\infty(0)<\infty,\\
e^{-a_\infty} & \text{if} \ X_\infty(0)=\infty. \end{cases}$$
In particular, the limiting probability is in $(0,1)$. Conditioned on $\{X_n(\tau_n^\star)=0\}$, the process behaves effectively as though it is in the linear diffusive regime with parameter $1-1/r_\infty$, and the same approximation applies. On the event $\{X_n(\tau_n^\star)=X_n^\star \}$, the process enters the metastable regime. This dichotomy appears to have been first observed in \cite{thresh-lim}.\\

\item \nid\tbf{Metastable}: In this case, the process $X_n$ reaches $X_n^\star$ relatively quickly, with typical fluctuations described as follows: the process $(X_n(t/\dlt_n)-X_n^\star)/\sqrt{n}$ converges in distribution to the Ornstein-Uhlenbeck process $Y$ described by
$$dY = -Ydt + \sqrt{2/r_\star}dB.$$
General results of this type can be found in the work of Kurtz \cite{Kurtz-DDMC}. As observed in \cite{thresh-lim}, when $r_n=r_\infty>1$ for each $n$, on each excursion from $X_n^\star$, the probability of hitting $0$ can be computed quite precisely, using well-known formulae for birth and death processes. In this article we extend these calculations to the whole supercritical phase, using essentially the same method but exercising greater care in the barely supercritical ($c_\infty=\infty$ and $r_\infty=1$) case.
\enumend

In stating the results for the subcritical phase it is helpful to define $\gma_n=1-r_n=-\dlt_n$. We now give precise statements of the results, take care to cite previous work in the relevant cases.

\begin{theorem}[Discrete case, Theorem 1 (B2) in \cite{thresh-lim}]\label{thm:disc}
Suppose $r_\infty \le 1$ and $X_\infty(0)\in \N$. For each $t>0$, as $n\to\infty$,

\begin{enumerate}[noitemsep]
\item if $r_\infty=1$ then $\P(\tau_n \le t) \to (1+1/t)^{-X_\infty(0)}$, and
\item if $r_\infty<1$ then $\P(\tau_n \le t) \to (1+\gamma_\infty/(e^{\gamma_\infty t}-1))^{-X_\infty(0)}$.
\end{enumerate}
\end{theorem}

Theorem \ref{thm:disc} is proved by first studying the extinction time of the linear birth-and-death process $Z_n$ from \eqref{eq:lbdp}, then transferring the result to $X_n$ using a coupling with the property that $\P(X_n(t)=Z_n(t) \ \text{for all} \ t \in [0,\tau_n])=1-o(1)$.\\

The next result appears not to have been proved yet in any of the listed references, though it follows straightforwardly from the coupling method already discussed from \cite{bright-luz}.

\begin{theorem}[Linear diffusive]\label{thm:lin-diff}
Suppose $r_\infty \le 1$, $X_\infty(0)=\infty$ and either
\enumalph
\item $c_\infty=-\infty$ and $a_\infty<\infty$, or
\item $c_\infty>-\infty$ and $a_\infty=0$.
\enumend

Then $\tau_n/X_n(0) \stackrel{(d)}{\to} H_{a_\infty}$ with
 $$\P(H_{a_\infty}\le w) = \begin{cases} e^{-1/w} & \text{if} \quad a_\infty=0, \\
                e^{-a_\infty/(e^{a_\infty w}-1)} & \text{if} \quad a_\infty>0.
               \end{cases}$$
\end{theorem}

Theorem \ref{thm:lin-diff} is proved in a similar way to Theorem \ref{thm:disc}. In this case, we cannot enforce $X_n(t)=Z_n(t)$ for all $t\in [0,\tau_n]$ with high probability, so instead we bound $X_n$ above and below by linear birth-and-death processes with different parameters. This ``sandwiching'' idea appears in \cite{bright-luz} to prove the final stage of Theorem \ref{thm:sub} (see below). A particular case of Theorem \ref{thm:lin-diff}, where $r=1+Cn^{-\alpha}$ for $\alpha <1/2$, is discussed in \cite{crit-scale}; they prove convergence to the diffusion limit in \eqref{eq:lin-diff}, but not the convergence of $\tau_n$. Although we do not do it here, we could also prove Theorem \ref{thm:lin-diff} via the diffusion limit \eqref{eq:lin-diff}, by the same methods we use to prove Theorem \ref{thm:nonlin-diff}.\\

The next result is the main result of Brightwell, House and Luczak \cite{bright-luz}.

\begin{theorem}[Subthreshold cutoff, Theorem 1.1 in \cite{bright-luz}] \label{thm:sub}
Suppose that $c_\infty=-\infty$ and $a_\infty=\infty$. Let $G$ denote the standard Gumbel: $\P(G\le w)=\exp(-e^{-w})$.\\

Then, $\gma_n \, \tau_n - g_n(X_n(0)) \cd G$ as $n\to\infty$, where
$$g_n(X) = \log (\gma_n^2 n) - \log (r_n + \gma_n n/X).$$
In particular,
$$g_n(X_n(0)) = \begin{cases} \log(\gma_n X_n(0)) + o(1) & \text{if} \quad X_n(0)=o(\gma_n n), \\
						 \log(\gma_n^2 n) - \log(r_\infty+1/b_\star) & \text{if} \quad X_n(0)/\gma_n n \to b_\star \in (0,\infty].\end{cases}$$
						 						
\end{theorem}

In \cite{bright-luz} the parameter $r$ is denoted $\lambda$, and the rate of $X\to X-1$ is equal to $\mu X$ instead of $X$; since $\mu$ can be set to $1$ by a uniform time change, no generality is lost. In \cite{bright-luz}, Theorem \ref{thm:sub} is proved in three stages, which we describe in the present notation.
\enumrom
\item \textit{initial stage}: a differential inequality for $x_n$ (incidentally, the same one originally proved in \cite{feller-1939german}) is used to show that from any initial value, $x_n$ drops to $\gma_n |c_n|^\ep$ for some $\ep>0$ within $o(1/\gma_n)$ amount of time (see their Lemma 4.2 and note that since we take $\mu=1$, in our notation $\mu-\lambda$ is just $\gma_n$, and since $\mu-\lambda \to 0$, $\lambda \to 1$ as $n\to\infty$).\\
\item \textit{intermediate stage}: from any initial value $x_n(0) \le \gma_n |c_n|^\ep$ until the first time that $x_n \le \gma_n |c_n|^{-\ep}$, $x_n$ remains close to the solution of the corresponding logistic equation $x'=rx(1-x)-x$. \\

\item \textit{final stage}: from any initial value $x_n(0)\le \gma_n |c_n|^{-\ep}$, a coupling argument is used to bound $X_n$ both above and below by linear birth-and-death processes with different parameters, as discussed earlier.
\enumend 

Our proof of Theorem \ref{thm:sub} in the initial and final stage is basically identical to the one in \cite{bright-luz}; in the intermediate stage we take a somewhat different approach, that we discuss in Section \ref{sec:sub}.

\begin{theorem}[Non-linear diffusive]\label{thm:nonlin-diff}

Suppose $c_\infty \in \R$ and $a_\infty>0$ and let $Y_n$ denote the rescaled process $Y_n(t) = X_n(\sqrt{n}\, t)/\sqrt{n}$. Let $y=\lim_{n\to\infty}Y_n(0)\in (0,\infty]$. Then, $Y_n$ converges in distribution in to the diffusion $Y$ that solves the SDE
$$dY = Y(c_\infty-Y)\, dt + \sqrt{2Y}dB, \quad Y(0) = y,$$
and $\tau_n/\sqrt{n} \cd T$, where $T=\inf\{t:Y(t)=0\}$.

\end{theorem}

The convergence of $X_n(\sqrt{n} \, t)/\sqrt{n}$ to the diffusion limit $Y$ is proved in \cite{crit-scale}, although they do not prove convergence of $\tau_n$. To obtain the latter we employ the continuous mapping theorem to obtain convergence of the hitting time of $\ep$ for small $\ep>0$, then show that once $Y_n \le \ep$, $Y_n$ is likely to hit $0$ in a short time.\\

The following result is proved in \cite{thresh-lim} in the particular case where $r_n=r_\infty>1$ for all $n$ , although the rapid extinction result is expressed without conditioning on hitting $0$ before $X_n^\star$, which is just a difference in the presentation. Our proof takes the same basic approach as \cite{thresh-lim}, with the following caveats:
\enumrom
\item the calculations become more delicate when $r_\infty =1$, requiring a finer analysis, and
\item we have found a couple of gaps in their proof of the exponential limit (see discussion below the statement of Theorem \ref{thm:meta}), so we took the steps described below to bridge those gaps.
\enumend

\begin{theorem}[Threshold and metastable]\label{thm:meta}
Suppose $c_\infty=\infty$ and $a_\infty>0$. Let $x_n^\star=1-1/r_n$, $X_n^\star = \lf n x_n^\star \rf$, $\tau_n^\star=\inf\{t\colon X_n(t) \in \{0,X_n^\star\}\}$, $A_n^\star = \{X_n(\tau_n^\star)=X_n^\star\}$ and $B_n^\star=(A_n^\star)^c = \{X_n(\tau_n^\star)=0\}$.

\enumar
\item Probability of rapid extinction:
$$\lim_{n\to\infty}\P(B_n^\star )=\begin{cases} r_\infty^{-X_\infty(0)} & \text{if} \ X_\infty(0)<\infty,\\
e^{-a_\infty} & \text{if} \ X_\infty(0)=\infty.\end{cases}$$
In particular $\P(A_n^\star)\to 1$ if $a_\infty=\infty$.\\
\item Scaling of $\tau_n$ on rapid extinction: if $a_\infty<\infty$ and
\enumalph
\item $r_\infty=1$ then for each $t\ge 0$,
$$\P(\tau_n/X_n(0) \le t \mid B_n^*) \to \P( H_{a_\infty} \le t)$$
as $n\to\infty$, as in the linear diffusive regime, and\\
\item if $r_\infty>1$ then conditioned on $(A_n^*)^c$, $\tau_n/r_n$ has the same limit as $\tau_n$ in \\ Case 2. of Theorem \ref{thm:disc}, except with $\dlt_\infty/r_\infty$ in place of $\gma_\infty$.\\
\enumend

\item Scaling of $\tau_n$ at metastability:
\enumalph
\item Expected time to extinction:
$$\E[ \, \tau_n \mid A_n^\star \, ]\sim \sqrt{\frac{2 \pi}{n}} \, \frac{r_n}{\dlt_n^2} \, e^{n \, (\log r_n + 1/r_n-1)}.$$
\item Exponential limit: for each $t\ge 0$, as $n\to\infty$
$$\P(\tau_n/\E[\tau_n \mid A_n^\star] \le t \mid A_n^\star  ) \to 1-e^{-t}.$$
\enumend
\enumend
\end{theorem}

Having access to the prefactors (the stuff in front of the exponential function) in the expected time to extinction allows us to see how it blends into the non-linear diffusive limit as $c_n$ approaches $O(1)$. For $r_n$ near $1$, a Taylor expansion gives $\log r_n+1/r_n-1 \sim \dlt_n^2/2$, so the exponential term is $\exp((1+o(1))n\dlt_n^2/2) = \exp((1+o(1))c_n^2/2)$ since $c_n=\sqrt{n}\dlt_n$, and the prefactor is $\sqrt{2\pi}r_n/(\sqrt{n}\dlt_n^2) = \sqrt{2\pi n}r_n/c_n^2$. Thus when $r_\infty=1$,
$$\E[\tau_n \mid A_n^\star] \sim \frac{\sqrt{2\pi n}}{c_n^2} e^{(1+o(1))c_n^2/2},$$
which is of order $\sqrt{n}$ when $c_n=O(1)$.\\

Much of the proof of Theorem \ref{thm:meta} revolves around precise estimation of various sums that all seem to involve the function
$$\nu(j,k) := \prod_{i=j+1}^k q_-(i)/q_+(i),$$
where $q_-(i) = i, \ q_+(i) = rX(1-X/n)$ is the rate at which $X_t$ decreases, respectively increases, by 1 when $X_t=i$. Let's take a moment to discuss the basic approach for each part:
\enumar 
\item \textit{Probability of rapid extinction:} An explicit formula for $h_+(j):= \P(A_n^\star \mid X_0=j)$ is available, in terms of the transition rates of the process, and can be estimated. Here is where we first encounter $\nu(j,k)$.\\

\item \textit{Scaling on $\tau_n$ on rapid extinction:} Using the Doob h-transform applied to the function $h_-:=1-h_+$, we study $X_n$ conditioned on $B_n^\star$. The transition rates can be written using $h_-$. By estimating transition rates we show the conditioned process corresponds to the setting of Theorem \ref{thm:disc} or \ref{thm:lin-diff}.\\
\item \textit{Scaling of $\tau_n$ at metastability:} Conditioning on $A_n^\star$, we break up $\tau$ into three epochs: the time to hit $X_n^\star$ (approach), the time spent on excursions that return to $X_n^\star$ (sojourn) and the last excursion from $X_n^\star$ to $0$ (fall). We show the approach and fall time are small compared to the sojourn time, in expectation; so far this is the same approach as in \cite{thresh-lim}. We then use a coupling argument that shows the process is ``forgetful'' in order to derive the exponential limit for the sojourn time; this method is a departure from \cite{thresh-lim}.
\enumend
We have also found the following gaps in the proofs in \cite{thresh-lim}:
\enumar

\item \textit{Exponential limit:} the exponential limit for $\tau_n/\E[\tau_n]$ is obtained in \cite{thresh-lim} via the same approach taken here: showing that the approach and fall time are small and that the sojourn time converges to exponential. Using their notation, for each $n$, the sojourn time can be written as a sum
$$\sum_{m=1}^{K_n} \sigma_n(m)$$
where $(\sigma_n(m))_{m=1}^{\infty}$ are i.i.d. and $K_n$ is geometric independent of $(\sigma_n(m))$, with parameter $1/\E[K_n]$ that $\to 0$ as $n\to\infty$. In claiming that the normalized sojourn time has an exponential limit, they appeal to Theorem 8.1A of Kielson \cite{rar-exp}. However, the theory discussed in Keilson, including that result, pertains only to a single i.i.d. sequence $(T_m)$ of random variables. On the other hand, in this case the fact that the process itself depends on $n$ implies that a doubly indexed sequence must be considered. I did some calculations (not included) that suggest that uniform integrability of $\sigma_1(n)$ with respect to $n$ is sufficient in order to adapt Keilson's result. However, in \cite{thresh-lim} this is not done.\\

\item \textit{Expected excursion length:} when expressing the time to extinction as the sum of approach, sojourn, and fall time, it must be assumed that during the sojourn, the process necessarily returns to $X_n^\star$ before hitting zero, which requires conditioning on $A_n^\star$. However, the formula used to compute the expected duration of each excursion from $X_n^\star$, namely $\E[\sigma_n(1)]$, in the proof of Lemma 2 in \cite{thresh-lim}, does not take account at all of conditioning on returning to $X_\star$. The answer comes out correct, since the conditioning does not have much effect, but it cannot be ignored outright.
\enumend

The rest of the paper is organized as follows. In Section \ref{sec:coup} we discuss some basic coupling results that are used a few times throughout the paper. Section \ref{sec:bbp} treats the linear birth-and-death process, essentially proving Theorems \ref{thm:disc} and \ref{thm:lin-diff} for the linearization of $X_n$. In Section \ref{sec:sub} we prove Theorems \ref{thm:disc} and \ref{thm:lin-diff}, and the ``final stage'' of Theorem \ref{thm:sub}. In Section \ref{sec:crit} we prove the rest of Theorem \ref{thm:sub}. In Section 6 we prove Theorem \ref{thm:nonlin-diff}. Finally in Section \ref{sec:super} we prove Theorem \ref{thm:meta}. In most of our proofs we suppress the dependence on $n$, to avoid writing subscripts everywhere. I'll point this out as we go along.

\section{Coupling of Birth and Death processes}\label{sec:coup}

Recall a birth and death (b-d) process is a right-continuous, continuous time Markov chain with state space $\N = \{0,1,2,\dots\}$ that jumps by $\pm 1$ at each transition. There is a natural way to construct such a process, or even multiple such processes with different transition rates, from all initial conditions on a single probability space. We will focus on b-d processes with state space $\{0,\dots,N\}$ for some $N$, since this is all that's needed for this paper. We begin with the case of a single process.

\subsection{Natural coupling for a single process}
A b-d process is defined by its transition rates. For $x \in \{0,\dots,N\}$ let $b(x),d(x)$ be the transition rate from $x$ to $x+1,x-1$ respectively, and assume that $d(0)=0$, and that $b(N)=0$ and define independent Poisson point processes $B(x),D(x)$ on $\R_+$ with respective rates $b(x),d(x)$. As a function of the collection $(B(x),D(x)\colon x \in \{0,\dots,N\})$ we shall define, for each $x \in \{0,\dots,N\}$, a process $(\Phi(x,t)\colon t \in [0,\infty))$ which is a copy of the b-d process with the given transition rates, and with initial value $x$ at time $0$.
Let $\Phi(x,0)=x$ and define $(t_i(x))_{i \le I(x)}$ and $(\Phi(x,t_i(x)))_{i \le I(x)}$ recursively by
\begin{align*}
&t_0(x)=0, \ \ t_{i+1}(x)= \inf \{t>t_i(x) \colon t \in B(\Phi(x,t_i(x))) \cup D(\Phi(x,t_i(x))), \\
&\Phi(x,t_{i+1}(x)) = \begin{cases} \Phi(x,t_i(x))+1 & \ \text{if} \ \ t_{i+1}(x) \in B(\Phi(x,t_i(x))) \\
\Phi(x,t_i(x))-1 & \ \text{if} \ \ t_{i+1}(x) \in D(\Phi(x,t_i(x)))\end{cases} \ \ \text{and} \ \ \\
& I(x) = \inf\{i \colon b(\Phi(x,t_i(x)))+d(\Phi(x,t_i(x)))=0\}.
\end{align*}
Then, for $i<I(x)$ let $\Phi(x,t) = \Phi(x,t_i(x))$ for $t \in [t_i(x),t_{i+1}(x))$, and if $I(x)<\infty$ let $\Phi(x,t)=\Phi(x,t_{I(x)}(x))$ for $t \in [t_{I(x)}(x),\infty)$. This defines the process on $[0,\zeta(x))$ where
$$\zeta(x):=\begin{cases}
\lim_{i\to\infty} t_i(x)) & \text{if} \ \ I(x)=\infty \\
\infty & \text{if} \ \ I(x)<\infty.\end{cases}$$
To verify that $\zeta(x)=\infty$ when $I(x)=\infty$, note that $\{t_i(x)\colon i \ge 1\} \subset U:= \bigcup_{y \le N}B(y) \cup D(y)$. Since $U$ is a Poisson point process with finite total intensity $\sum_{y \le N}b(y)+d(y)$, with probability 1, $U_N \cap [0,T]$ is finite for each fixed $T>0$, which implies that $\lim_{i\to\infty}t_i$ must be infinite.\\

Having constructed the coupling, we verify the following desirable properties.

\begin{lemma}\label{lem:nat-coup}
Let $\Phi(x,t)$ be as defined above. If $\Phi(x,s)=\Phi(y,s)$ then $\Phi(x,t)=\Phi(y,t)$ for all $t\ge s$, and if $x \le y$ then $\Phi(x,t) \le \Phi(y,t)$ for all $t\ge 0$.
\end{lemma}

\begin{proof}
Suppose $\Phi(x,s)=\Phi(y,s)$ that we denote by $z$. Then for some $i,j$, $s\ge \max(t_i(x),t_j(y))$ and $\Phi(x,t_i(x))=\Phi(y,t_j(y))=z$. If $b(z) + d(z)=0$ then $I(x)=i$ and $I(y)=j$, and $\Phi(x,t) = \Phi(y,t) = z$ for all $t\ge s$. Otherwise, $I(x)>i$, $I(y)>j$ and $\min(t_{i+1}(x),t_{j+1}(y))>s$. We then have $t_{i+1}(x)=t_{j+1}(y)$ since both are equal to $\inf\{t>s \colon t \in B(z) \cup D(z)\}$, and the construction ensures that $\Phi(x,t)=\Phi(y,t)$ for all $t\ge s$.\\

For the second statement, first note that since $(B(x)\cup D(x)\colon x \in \{0,\dots,N\})$ are independent, from standard properties of Poisson point processes it follows that if $x \ne y$ then $B(x)\cup D(x)$ and $B(y) \cup D(y)$ are disjoint. It follows from the construction that if $\Phi(x,t^-) \ne \Phi(y,t^-)$ then either $\Phi(x,t)=\Phi(x,t^-)$ or $\Phi(y,t) = \Phi(y,t^-)$, i.e., both cannot jump simultaneously. Since $\Phi(x,0)-\Phi(y,0)$ is integer-valued and since
\itemgo
\item for any $x$, $t\mapsto \Phi(x,t)$ is piecewise constant,
\item for any $x$ and fixed $T>0$, the set $\{t \in [0,T]\colon \Phi(x,t) \ne \Phi(x,t^-)\}$ is a.s.~finite, and
\item for any $x,t$, $|\Phi(x,t)-\Phi(x,t^-)| \in \{-1,0,1\}$,
\itemend 
if $x \le y$ and $\Phi(x,t)>\Phi(y,t)$ then since $\Phi(x,0)=x$ and $\Phi(y,0)=y$, for some $s<t$, $\Phi(x,s)=\Phi(y,s)$. Using the first statement, we then have $\Phi(x,t)=\Phi(y,t)$, contradicting $\Phi(x,t)>\Phi(y,t)$.
\end{proof}

\subsection{Natural coupling for two or more ordered processes}

Next we describe a similar construction for multiple b-d processes with different transition rates, on a common state space $\{0,\dots,N\}$. Index the processes $1,\dots,k$ and let $b_i(x),d_i(x)$ denote the transition rates. We will consider only the case in which rates are ordered such that if $i<j$ then $b_i(x) \le b_j(x)$ and $d_i(x) \ge d_j(x)$ for each $x$. Let $\beta_1(x) = b_1(x)$ and $\dlt_k(x)=d_k(x)$, then for $1 < i \le k$ let $\beta_i(x) = b_i(x) - b_{i-1}(x)$ and for $1\le i< k$ let $\dlt_i(x) = d_i(x)-d_{i+1}(x)$. Define Poisson point processes $B_i(x),D_i(x)$ on $\R_+$ with respective rates $\beta_i(x),\dlt_i(x)$. Then, for each $i\in \{1,\dots,k\}$ define $(\Phi_i(x,t)\colon t\in [0,\infty))$ in the same way as $\Phi$ from the previous subsection, but using $\bigcup_{j=1}^i B_j(x)$ and $\bigcup_{j=i}^k D_j(x)$, and $\sum_{j=1}^i b_j(x)$ and $\sum_{j=i}^k d_j(x)$, in place of $B(x)$ and $D(x)$, and $b(x)$ and $d(x)$, respectively. Then $\Phi_i(x,t)$ is a copy of the b-d process with transition rates $b_i,d_i$ and initial value $x$.\\

For fixed $i$, the above construction amounts to the same as in the previous subsection, so Lemma \ref{lem:nat-coup} applies to $\Phi_i$. Another useful property is summarized in the following result.

\begin{lemma}\label{lem:nat-coup2}
Let $\Phi_i(x,t)$, $i=1,\dots,k$, $x \in \{0,\dots,N\}$, $t\in[0,\infty)$ be as defined above. \\If $x\le y$ and $i\le j$ then $\Phi_i(x,t) \le \Phi_j(y,t)$ for all $t\ge 0$.
\end{lemma}

\begin{proof}
For the same reason as in the proof of Lemma \ref{lem:nat-coup}, if $x \le y$ and $\Phi_i(x,t) > \Phi_j(y,t)$ then for some $s<t$, $\Phi_i(x,s)=\Phi_j(y,s)$. Moreover, $s$ can be chosen so that in addition, for some $u\in (s,t]$, $\Phi_i(x,u^-)=\Phi_j(y,u^-)=:z$ and either
\itemgo
\item $\Phi_i(x,u)=z+1$ and $\Phi_j(y,u)=z$ or
\item $\Phi_i(x,u)=z$ and $\Phi_j(y,u)=z-1$.
\itemend
The first case implies that $u \in \bigcup_{m=1}^i B_m(z)$ and $u \notin \bigcup_{m=1}^j B_m(z)$ which is impossible since $i \le j$. Similarly, the second case implies $u\in \bigcup_{m=j}^k D_m(z)$ and $u\notin \bigcup_{m=i}^k D_m(z)$ which again is impossible.
\end{proof}

\section{Linear birth-and-death process}\label{sec:bbp}
The linear birth-and-death process $\Z_+$ with parameter $r$ is defined by the transitions
$$Z \to \begin{cases}Z+1 & \text{at rate} \quad r Z \\
         Z-1 & \text{at rate} \quad Z.
        \end{cases}$$
$Z$ can be thought of as the number of cells in a process in which each cell independently dies at rate $1$ and splits into two cells at rate $r$. For this process we are interested in how the extinction time $\tau=\inf\{t:Z(t)=0\}$ scales with $r$ and $Z_0$, in cases where extinction is asymptotically certain, i.e., $\P(\tau<\infty)\to 1$. The results are by now routine, but as I could not find a reference that states them all together, and since they are easy to prove, I have included the proof.

\begin{theorem}\label{thm:bbp}
Let $Z_n$ denote a sequence of copies of the above process, with respective initial condition and parameter $Z_n(0),r_n$. Let $\tau_n=\inf\{t\colon Z_n(t)=0\}$, $\gamma_n = 1-r_n$ and $a_n=\gma_n Z_n(0)$. Suppose that $Z_n(0) \to Z_\infty(0) \in \N \cup \{\infty\}$, $r_n\to r_\infty \le 1$ and $a_n\to a_\infty \in [0,\infty]$. Let $\gma_\infty=\lim_{n\to\infty}\gma_n$.

\begin{enumerate}
\item Suppose $Z_\infty(0)<\infty$ and fix any value of $t\ge 0$.
\begin{enumerate}[noitemsep]
\item If $r_\infty=1$ then $\P(\tau_n \le t) \to (1+1/t)^{-Z_\infty(0)}$.
\item If $r_\infty<1$ then $\P(\tau_n \le t) \to (1+\gamma_\infty/(e^{\gamma_\infty t}-1))^{-Z_\infty(0)}$.
\end{enumerate}
\item Suppose $Z_\infty(0)=\infty$.
\begin{enumerate}
 \item If $a_\infty \in [0,\infty)$ then $\tau_n/Z_n(0) \stackrel{(d)}{\to} H_{a_\infty}$ with
 $$\P(H_{a_\infty}\le w) = \begin{cases} e^{-1/w} & \text{if} \quad a_\infty=0, \\
                e^{-a_\infty/(e^{a_\infty w}-1)} & \text{if} \quad a_\infty>0.
               \end{cases}$$
 \item If $a_\infty=\infty$ and $a_n=b_n+o(b_n/\log b_n)$ for some sequence $(b_n)$ such that $b_n\to\infty$ then $\gamma_n \, \tau_n - \log b_n \stackrel{(d)}{\to} G$, where $G$ has $\P(G\le w)=e^{-e^{-w}}$ for $w\in \R$.
\end{enumerate}
\end{enumerate}
\end{theorem}

\begin{proof}[Proof of Theorem \ref{thm:bbp}]
For compactness of notation we'll suppress the dependence on $n$ in all variables and write $Z_t$ instead of $Z(t)$. In Chapter III of \cite{athreyabp} a more general model is considered in which each particle independently dies at some rate $\alpha$ (they call it $a$ but that notation's already in use for us) and is replaced with $k$ particles with probability $p_k$; the present model corresponds to $\alpha =1+r$ and $p_0=1/(1+r)$, $p_2=r/(1+r)$. Using the Kolmogorov backward equation of the process (equation (4) of III.2), for $\rho(t)=\P(Z_t=0 \mid Z_0=1)$ they obtain the differential equation (equation (2) of III.4)
$$\rho' = \alpha (-\rho + \sum_{k=0}^{\infty}p_k \rho^k).$$
In our case this gives $\rho' = -(1+r)\rho + 1 + r\rho^2$, which can be conveniently factored to give
\begin{align}\label{eq:bbp-rho}
\rho' = (1-r\rho)(1-\rho), \quad \rho_0=0.
\end{align}
Separating variables,
$$\frac{d\rho}{(1-r\rho)(1-\rho)} = dt.$$
Letting $\gamma=1-r$ we note that
$$\frac{1}{1-\rho} - \frac{r}{1-r\rho} = \frac{\gamma}{(1-r\rho)(1-\rho)}.$$
Solving the DE and using $\rho(0)=0$,
$$\gamma \, t = \log\left(\frac{1-r\rho(t)}{1-\rho(t)}\right).$$
If $\gamma \ne 0$, solving for $\rho(t)$ then gives
\begin{align}\label{eq:rhot}
\rho(t)^{-1} = \frac{e^{\gamma t}-r}{e^{\gamma t}-1} = 1 + \gamma / (e^{\gamma t}-1).
\end{align}
Part 1 of Theorem \ref{thm:bbp} then follows easily; the case $r_\infty=1$ follows by taking the limit of $\rho(t)$ as $\gamma \to 0$, or it could be computed directly by solving the ODE in the special case $r=1$.\\

We now tackle the second part of Theorem \ref{thm:bbp}, using \eqref{eq:rhot} to determine $\P(\tau \le t \mid Z_0) = \rho(t)^{Z_0}$ when $Z_0 \to \infty$, under various limits of $\gamma Z_0$. As $Z_0\to\infty$, $\rho(t)^{Z_0}\to e^{-c}$ if $\gamma/(e^{\gamma t}-1)=c/Z_0$. Solving this gives
$$\gamma \, t = \log (1 + \gamma Z_0/c).$$
Recalling that $a=\gamma Z_0$, there are two cases.\\

\noindent\textit{Case 1: $a \to a_\infty \in [0,\infty)$.} In this case, if $\rho(t)^{Z_0} \to e^{-c}$ then $t = (Z_0/a)(\log(1 + a/c))$.\\
Setting $w=\frac{1}{a}\log(1+a/c)$ gives $c = a/(e^{aw}-1)$. If $a \to 0$ then $a/(e^{aw}-1) \to 1/w$ and
$$\P(\tau \le w Z_0) \to e^{-1/w},$$
while if $a \to a_\infty>0$ then
$$\P(\tau \le w Z_0) \to e^{-a_\infty/(e^{a_\infty w}-1)}.$$

\noindent\textit{Case 2: $a\to\infty$.} Taking $\gamma \, t = \log a - \log c$ gives $\rho(t)^{Z_0} \to e^{-c}$ as $Z_0\to\infty$.\\
Letting $w=-\log c$ so that $c=e^{-w}$,
$$\P(\tau \le \gamma^{-1}(\log \gamma Z_0 + w) ) \to e^{-e^{-w}} \quad \text{as} \quad Z_0 \to \infty.$$
In other words, $\gamma \tau - \log \gamma Z_0$ has a standard Gumbel distribution.\\
The following short lemma concludes the proof of case 2, and thus the proof of Theorem \ref{thm:bbp}.
\end{proof}

\begin{lemma}\label{lem:bpp}
Let $Z_n$ be a sequence of copies of the linear birth-and-death process with initial value $Z_n(0)\to\infty$ and parameter $r_n=1-a_n/Z_n(0)$ with $a_n\to\infty$. Let $(b_n)$ be a sequence with $b_n\to\infty$ and let $\tau_n=\inf\{t:Z_n(t)=0\}$. Then
$$\lim_{n\to\infty}\P(\tau_n \le \frac{Z_n(0)}{b_n}(\log b_n + w)) \to e^{-e^{-w}} \,\, \text{for all} \,\, w \in \R \quad \,\, \text{iff} \,\, \quad a_n-b_n = o(b_n/\log b_n).$$
\end{lemma}

\begin{proof}
As before, suppress the dependence on $n$. Fix $w \in \R$ and let $v$ be such that
$$\frac{1}{b}( \log b + w) = \frac{1}{a}(\log a + v).$$
Since $\log a + v = \displaystyle\frac{a}{b}(\log b + u)$, subtracting $\log a+w$ from both sides gives
\begin{align}\label{eq:bpp1}
v-w=\log(b/a) + \frac{a-b}{b}\log b + \frac{a-b}{b}w.
\end{align}
If $a-b=o(b/\log b)$ the second term is $o(1)$. Since $b\to\infty$, $a = b + o(b)$ so $\log(b/a)=o(1)$, $\displaystyle\frac{a-b}{b}w=o(1)$ and so $v-w=o(1)$. 
On the other hand, suppose that $v-w=o(1)$ for every $w\in \R$. Setting $w=0$ gives
\begin{align}\label{eq:bpp2}
\log(b/a) + \frac{a-b}{b}\log b = o(1). 
\end{align}
Then, setting $w=1$ and using \eqref{eq:bpp1} and \eqref{eq:bpp2} gives $\displaystyle\frac{a-b}{b}=o(1)$. 
This in turn implies that $\log(b/a)=o(1)$. Using \eqref{eq:bpp2} once more gives $\displaystyle\frac{a-b}{b}\log b=o(1)$ as required.
\end{proof}

\section{Approximation by linear birth-and-death processes}\label{sec:lbdp-approx}

In this section we use the method of Brightwell, House and Luczak \cite{bright-luz}, namely, approximation by linear birth-and-death processes, to prove Theorem \ref{thm:disc}, \ref{thm:lin-diff} and the ``final stage'' part of Theorem \ref{thm:sub}. As in the previous section, we shall suppress dependence on $n$ and write $X_t,Z_t$ etc.\\

For $M > X_0>0$ (both depending on $n$) to be determined, let $r'=r(1-M/n)$ and let $Z,Z'$ be the linear birth and death processes with respective parameters $r,r'$ and common initial value $X_0$. Let $\tau_M = \inf\{t:X_t \in \{0,M\}\}$. Then, applying the natural coupling of Section \ref{sec:coup} to $Z',X,Z$ labelled $1,2,3$ respectively and using Lemma \ref{lem:nat-coup2} gives $Z'_t \le X_t \le Z_t$ for $t\le \tau_M$. 
Let $\tau_Z = \inf\{t:Z_t=0\}$ and $\tau_{Z'}= \inf\{t:Z_t'=0\}$.
The goal is to take $M$
\begin{enumerate}[noitemsep]
\item large enough that $\P(X_{\tau_M}=M) = o(1)$ and
\item small enough that
\begin{enumerate}[noitemsep]
\item $\P(X_t=Z_t$ for $t\le \tau_M)=1-o(1)$ if $X_0$ is fixed,
\item $\tau_Z$ and $\tau_{Z'}$ have a common rescaled limit if $X_0\to\infty$.
\end{enumerate}
\end{enumerate}
We begin with a general observation. Let $p_-(X),p_+(X)$ denote the probability that $X\to X-1$, respectively, $X\to X+1$ in the embedded discrete time Markov chain, or jump chain, of $X_t$. Then $p_-(X)=1/(1+r(1-X/n))$ and $p_+(X)=r(1-X/n)/(1+r(1-X/n))$ and in particular, $p_-(X)/p_+(X) = 1/(r(1-X/n))\ge 1/r$ for all $X$, so it follows easily that $r^{-X_t}$ is a supermartingale. Recall that $\tau = \inf\{t \colon X_t=0\}$. Using optional stopping, which is applicable since $X$ is bounded and $\P(\tau<\infty)=1$, it is easy to show that
\begin{align}\label{eq:X-sm}
\P(\tau \ne \tau_M \mid X_0) = \P(X_{\tau_M} = M \mid X_0) \le \frac{r^{-X_0} - 1}{r^{-M}-1}.
\end{align}
If $r < 1$ then using the estimate $a/b \le (a+c)/(b+c)$ that holds for $0<a\le b$ and $c>0$, we find that if $X_0,M>0$ then
\begin{align}\label{eq:taum-1}
\frac{r^{-X_0}-1}{r^{-M}-1} \le \frac{r^{-X_0}}{r^{-M}} = (1-\gamma)^{M-X_0} \to 0 \quad \text{if} \quad \gamma \, (M-X_0)\to\infty.
\end{align}
Recall $\gamma = 1-r$. If $\gamma \to 0$ and if, a fortiori, $\gamma M \to 0$ then
\begin{align}\label{eq:taum-2}
\frac{r^{-X_0}-1}{r^{-M}-1} = \frac{(1-\gamma)^{-X_0}-1}{(1-\gma)^{-M}-1} \sim \frac{X_0}{M} \to 0 \quad \text{if} \quad X_0=o(M).
\end{align}

We proceed by cases. The first case is Theorem \ref{thm:disc}.

\begin{proof}[Proof of Theorem \ref{thm:disc}]
In this setting, $\gma\to\gma_\infty\ge 0$ and we can assume $X_0$ is constant. For any $t \ge 0$,
\begin{align*}
& |\P(\tau \le t \mid X_0) - \P(\tau_Z \le t \mid X_0)| \\
& \le \P(\tau \ne \tau_M \mid X_0) + \P(\tau = \tau_M \ \ \text{and} \ \ \tau_M \wedge t \ne \tau_Z \wedge t \mid X_0).\end{align*}
Theorem \ref{thm:disc} then follows from Theorem \ref{thm:bbp} if $M$ can be chosen so that for any $t\ge 0$, the above $\to 0$ as $n\to\infty$.\\
We let $p_1 = \P(\tau \ne \tau_M\mid X_0)$ and $p_2 = \P(\tau = \tau_M \ \ \text{and} \ \ \tau_M \wedge t \ne \tau_Z \wedge t \mid X_0)$.\\
We first find $M$ such that $p_1 \to 0$. If $\gma_\infty > 0$, using \eqref{eq:taum-1} we find that $p_1 \to 0$ if $M\to\infty$. \\
If $\gma_\infty = 0$, then using \eqref{eq:taum-2} we find that $p_1 \to 0$ if $M \to \infty$ and $\gma M \to 0$.\\
We next find $M$ such that $p_2\to 0$. If $\tau=\tau_M$ then since $X,Z$ are both absorbed at $0$,
$$p_2 \ \le \P(X_s \ne Z_s \ \ \text{for some} \ \ s \le t \wedge \tau_M \mid X_0).$$
If $X,Z \in [0,M]$ their transition rates differ by at most $(r-r')M = rM^2/n$, so using the exponential distribution and noting $e^{-x} \ge 1-x$ we find that
$$\P(X_s \ne Z_s \ \ \text{for some} \ \ s \le t \wedge \tau_M \mid X_0) \le 1 - e^{-rM^2t/n} \le rM^2t/n.$$
Since $r$ is bounded in $n$, the right-hand side $\to 0$ for fixed $t$ provided $M=o(\sqrt{n})$.\\
Thus both $p_1,p_2 \to 0$ if we take $M\to\infty$ sufficiently slowly.\\
\end{proof}

Next we prove Theorem \ref{thm:lin-diff}.
\begin{proof}[Proof of Theorem \ref{thm:lin-diff}]
We first re-interpret somewhat the conditions of the theorem.
\enumalph
\item \textit{first option:} $c_\infty=-\infty$ and $a_\infty<\infty$. Since $c_\infty=-\infty$, $1/\sqrt{n}=o(|\dlt_n|)$ so $a_\infty = \lim_{n\to\infty} |\dlt_n| X_n(0)$. Since $a_\infty<\infty$, $X_n(0)=O(1/|\dlt_n|) = o(\sqrt{n})$. In shorthand: $\gma X_0 \to a_\infty\in [0,\infty)$ and $X_0=o(\sqrt{n})$.\\
\item \textit{second option:} $c_\infty>-\infty$ and $a_\infty=0$. If $c_\infty \in \R$ then $a_\infty=0$ is equivalent to $X_n(0)/\sqrt{n} \to 0$, and implies $|\dlt_n| X_n(0) \to 0$. If $c_\infty=\infty$ then in the same way as above, $a_\infty =\lim_{n\to\infty} |\dlt_n| X_n(0)$ and $X_n(0)=o(\sqrt{n})$. So in this case as well, $\gma X_0 \to a_\infty \in [0,\infty)$ and $X_0=o(\sqrt{n})$.
\enumend
Thus the theorem is proved, if it can be proved when $\gma X_0 \to a_\infty \in [0,\infty)$ and $X_0=o(\sqrt{n})$. First note that with $p_1$ as in the proof of Theorem \ref{thm:disc},
\begin{align}\label{eq:tau-comp}
|\P(\tau_Z' \le \tau \le \tau_Z \mid X_0)| \ge 1 - p_1.
\end{align}

Since $X_0\to\infty$ and $\gma X_0 \to a_\infty<\infty$, $\gma \to 0$. For $W \in \{X,Z,Z'\}$, and writing $\tau$ as $\tau_X$, let $F_W(t) = \P(\tau_W/X_0 \le t)$. From \eqref{eq:tau-comp},
\begin{align}\label{eq:taus-ord}
F_Z(t) - p_1 \le F_X(t) \le F_{Z'}(t)+p_1.
\end{align}
By Theorem \ref{thm:bbp}, if $\gma X_0, \, \gma' X_0 \to a_\infty$ then for all $t\ge 0$, $F_Z(t),F_{Z'}(t) \to \P(H_{a_\infty} \le t)$.\\
Since $\gma \, X_0 \to a_\infty$ by assumption, it is enough to find $M$ such that i) $p_1\to 0$ and ii) $(\gma' - \gma)\, X_0 \to 0$.\\

\noindent For i), using \eqref{eq:taum-1} and \eqref{eq:taum-2} we need either $\gma\, M \to \infty$, or $X_0/M \to 0$ and $\gma\, M \to 0$.\\
For ii) we compute $(\gma'-\gma)X_0 = (r-r')X_0 = rX_0M/n$, so it is enough that $X_0M/n \to 0$.\\

\noindent\textit{Subcase 1: $a_\infty=0$.} Define $\beta = |\gma| \vee 1/\sqrt{n}$ and let $M = \beta^{-1}(\beta X_0)^{1/2}$. Since $\gma X_0 \to 0$ and $X_0=o(\sqrt{n})$, $\beta X_0 = \max(|\gma| X_0,\,X_0/\sqrt{n}) \to 0$, so $X_0/M = (\beta X_0)^{1/2} \to 0$ and $\gma \, M \le \beta \, M = (\beta X_0)^{1/2} \to 0$, satisfying i). Since $\beta \ge 1/\sqrt{n}$, $\beta^{-1} \le \sqrt{n}$ so $M=o(\sqrt{n})$ and $X_0M/n \to 0$, satisfying ii).\\

\noindent\textit{Subcase 2: $a_\infty>0$.} Since $X_0=o(\sqrt{n})$ and $\gma X_0$ has a positive limit, $\sqrt{n} \gma \to \infty$. Let $M = \gma^{-1}(\sqrt{n}\gamma)^{1/2}$. Then $\gma M = (\sqrt{n}\gamma)^{1/2}\to\infty$, satisfying i). Since $X_0=o(\sqrt{n})$ and $M = (\sqrt{n}/\gma)^{1/2} = \sqrt{n}/(\sqrt{n}\gma)^{1/2} = o(\sqrt{n})$, $X_0M/n \to 0$, satisfying ii).
\end{proof}

Now, we prove Theorem \ref{thm:sub} in the ``final stage'', i.e., when $X_n(0)\le n\gma_n |c_n|^{-\ep}$ for some $\ep>0$. The proof is broadly the same as  in \cite{bright-luz}, although here the result is applicable to a somewhat larger set of initial values.
\begin{proof}[Proof of Theorem \ref{thm:sub} in the case where $X_n(0) \le n\gma_n |c_n|^{-\ep}$ for some $\ep>0$]

Using our abbreviated notation, we shall prove the result under the slightly less restrictive condition $X_0\log(\gma X_0) = o(\gma n)$. To see that this is less restrictive, suppose $X_0 \le \gma n|c|^{-\ep}$ so that $X_0|c|^\ep \le \gma n$. By assumption in Theorem \ref{thm:sub}, $|c|\to\infty$ and $a=\gma X_0 \to \infty$. Using $|c|\to\infty$, $\log|c|=o(|c|^\ep)$ so $X_0\log|c| = o(X_0|c|^\ep)= o(\gma n)$. Using $a\to\infty$ and $X_0=o(\gma n)$, $\log(\gma X_0) = o(\log(\gma^2 n)$. Since $|c|=\gma\sqrt{n}$, $\log(\gma X_0)=o(\log(|c|)$. Therefore $X_0 \log (\gma X_0) = o(X_0\log|c|) = o(\gma n)$ as desired.\\

If $X_0\log(\gma X_0)=o(\gma n)$ then since $\gma X_0 \to \infty$, $X_0=o(\gma n)$. As noted in the statement of Theorem \ref{thm:sub}, in this case $g = \log(\gma X_0)+o(1)$. So, we want to show that $\gma \tau - \log(\gma X_0)$ converges to standard Gumbel.\\

Let $a =\gma X_0$ and for $W \in \{X,Z,Z'\}$ let $F_W(t) = \P(\gma \, \tau_W - \log a \le t)$. Then, \eqref{eq:taus-ord} also holds for this choice of $F_W(t)$. According to Theorem \ref{thm:bbp}, for all $t\ge 0$ $F_Z(t) \to \P(G\le t)$ and if $(\gma'-\gma)X_0 = o(a/\log a)$ then for all $t\ge 0$, $F_{Z'}(t) \to \P(G \le t)$. Following again the logic of the proof of Theorem \ref{thm:lin-diff}, it is enough to find $M$ such that i) $p_1\to 0$ and ii) $(\gma' - \gma)\, X_0 = o(a/\log a)$.\\

\noindent For i) we need $\gma \, (M-X_0)\to\infty$, for which $M = 2X_0$ suffices (and which, assuming only $\gma X_0\to\infty$, cannot be improved beyond a factor of $2$). For ii) we need $X_0M/n = o(\gma X_0/\log(\gma X_0))$. Using $M=2X_0$, the condition becomes $X_0 = o(\gma n/\log(\gma X_0))$, which is the condition given.

\end{proof}

\section{Subthreshold cutoff}\label{sec:sub}

In this section we prove Theorem \ref{thm:sub}, in the case where $X_n(0) \ge n\gma_n |c_n|^{-\ep}$ for any $\ep>0$. This corresponds to the intermediate and initial stages as described in \cite{bright-luz}; the final stage, where $X_n(0)\le n\gma_n |c_n|^{-\ep}$, is proved just above in the previous section. The proof for the initial stage is about identical to their proof. For the intermediate stage, the setup of the problem follows their approach, then I use a different method to handle the error term between the process and its deterministic approximation (denoted $e_N$ in their paper and $W$ here). In their paper, an auxiliary result (their Lemma 3.2) is used to show the maximum of the error term is (deterministically) bounded by $2$ times the maximum of the compensator of $X$, then the latter is estimated using the corresponding exponential martingale. Here, we compute the drift and diffusivity of the error term, then after making a time change, use a so-called drift barrier estimate, Lemma \ref{lem:driftbar} which is proved in \cite{SIS-SDE}, to show the error term remains small on the desired time interval. Neither approach seems to be strictly simpler or more efficient than the other.
\begin{proof}[Proof of Theorem \ref{thm:sub} in the case where $X_n(0) \ge n\gma_n |c_n|^{-\ep}$]
As before, suppress $n$ from the notation and write $X_t,Y_t$ etc when it is convenient. For this proof only, let $c=\sqrt{n}\gma$ which amounts to a change of sign; this saves us from always writing $|c|$. Fix a small $\ep>0$ to be determined and let $Y_t=X(t/\gma)/(\gma n)$. Then $\gma X_0=c^2Y_0$. Let's re-write $g_n(X_n(0))$ from the statement of Theorem \ref{thm:sub} in terms of $Y$. Abusing notation a bit, we have
$$g(Y_0) = 2\log c - \log(r + 1/Y_0).$$
Let $t_\star = \inf\{t \colon Y_t \le c^{-\ep}\} = \inf\{t\colon X(t/\gma) \le n\gma  c^{-\ep}\}$. The final stage corresponds to $t_\star=0$. Since $X$ jumps by $\pm 1$, $Y$ jumps by $\pm 1/\gma n = \pm 1/c\sqrt{n} = o(1/c)$, so
$$g(Y_0)-g(Y(t_\star)) = -\log(r+1/Y_0) + \log(c^{-\ep} \pm o(1/c)) = -\log(r+1/Y_0) - \ep \log c + o(1).$$
Thus to prove the result it remains to show that
\begin{align}\label{eq:t-st-frml}
t_\star - (-\log(r+1/Y_0)-\ep \log c) \to 0.
\end{align}
\noindent\textbf{Intermediate stage.} Suppose $c^{-\ep} \le Y_0 \le c^{\ep}$. We will use the notation for drift and diffusivity discussed in the Appendix. Using the transition rates from \eqref{eq:logis-model} in \eqref{eq:mc-dd} and noting $\gma=-(r-1)$, for $X_t$ we compute
\begin{align*}
\mu(X) &= rX(1-X/n)-X = -\gma X - rX^2/n, \\
\sigma^2(X) &=  (1+r(1-X/n))X \le (1+r)X.\end{align*}
From \eqref{eq:mc-dd} we infer that if $Y(t)=\alpha X(\beta t)$ then $\mu(Y) = \alpha\beta \mu(X)$ and $\sigma^2(Y) = \alpha^2\beta \sigma^2(X)$. So,
$$\mu(Y) = \frac{1}{\gma^2 n}\mu(X) = -\frac{1}{\gma n}X - r\frac{1}{(\gma n)^2}X^2 = - Y - rY^2$$
and using $\gma^2n=c^2$,
$$\sigma^2(Y) = \frac{1}{\gma^3n^2}\sigma^2(X) \le \frac{1}{\gma c^2}(1+r)X = \frac{1}{c^2}(1+r)Y.$$
As in \cite{bright-luz}, we'll directly estimate the distance between $Y$ and its deterministic approximation. Let $y(t)$ denote the solution to the initial value problem
$$y' = -y - ry^2, \quad y(0)=Y_0$$
and let $t^{\pm} = \inf\{t : y(t) = (1 \pm c^{-\ep})c^{-\ep} \}$. Solving by separation of variables,
\begin{align*}
t^{\pm} &= \log \frac{Y_0}{1+rY_0} - \log\frac{(1 \pm c^{-\ep})c^{-\ep}}{1+r(1 \pm c^{-\ep})c^{-\ep}} \\
&= -\log(r+1/Y_0) - \ep \log c + o(1).
\end{align*}
Thus, $t^\pm$ both have the desired limit for $t_\star$ as in \eqref{eq:t-st-frml}. The result will be proved if we show that $t^- \le t_\star \le t^+$ with probability $1-o(1)$. 
Define the error process $W_t=Y_t-y(t)$. If $\sup_{t \le t^-}|W_t| \le c^{-2\ep}$ then $t^- \le t_\star \le t^+$, so we will show the former has probability $1-o(1)$. We first compute the drift and diffusivity of $W$. Factoring the difference of squares,
$$\mu(W) = \mu(Y) - y' = -Y-rY^2+y-ry^2 = -W(1+r(Y+y)).$$
Since $(y(t))$ is continuous and has finite variation it has zero quadratic variation, so
$$\sigma^2(W) = \sigma^2(Y) \le \frac{1}{\gma^2}(1+r)Y.$$
Since $r,y,Y>0$,
\begin{equation}\label{eq:msratio}
\sgn(\mu(W)) = -\sgn(W) \quad \text{and} \quad \frac{|\mu(W)|}{\sigma^2(W)} \ge \frac{c^2r}{1+r}|W|.
\end{equation}
Next we change time from $t$ to $s$ such that $\mu_s(W) = -W_s$. To do so let
$$s(t) = \int_0^t (1 + r(Y_r+y(r))dr.$$
Since $\mu(W)$ and $\sigma^2(W)$ are both scaled by $dt/ds \in (0,1]$, \eqref{eq:msratio} remains valid after the time change, so
$$\mu_s(W) = -W_s \quad \text{and} \quad \sigma^2_s(W) = O(1/c^2).$$
Also, if $|W_s| \le c^{-2\ep}$ then $Y_s \le y(s) + c^{-2\ep} \le y(0) + c^{-2\ep} \le c^\ep+c^{-2\ep}$, so $s'(t) \le 1 + r(2c^{\ep} + c^{-2\ep}) = O(c^{\ep})$. Since $t^- = O(\log c)$, if $\sup_{s \le s(t^-)}|W_t| \le c^{-2\ep}$ then $s(t^-) = O(c^{\ep}\log c)=o(c)$ so it's enough to show that
$$\sup_{s \le c}|W_s| \le c^{-2\ep}.$$
We first give an upper bound on $W$. Next we use Lemma \ref{lem:driftbar} which is proved in \cite{SIS-SDE}. In the notation of Lemma \ref{lem:driftbar}, let $x = c^{-2\ep}/2$, $X = W - x$, $\Delta_\infty(X) = 1/\gma n$, $\mu_\star = x$, $\sigma^2_\star=C/c^2$ for some constant $C$, $C_\mu = 2x$ and since $\Delta_\infty(X)\mu_\star/\sigma_\star^2 = c^2 x/C\gma n = \gma^2n x/C\gma n \le \gma/C=O(1)$, take $C_\Delta$ to be some large enough constant. Then $\Gamma = \exp(x^2/16\sigma^2_\star) = \exp(c^{2-2\ep}/16 C) \ge c$ for large enough $c$, and since $c\to\infty$ by assumption, with probability $1-o(1)$, $X_s \le x$ or equivalently $W_s \le 2x = c^{-2\ep}$ for all $s \le c$. A matching bound for $-W$  is proved in the same way.\\

\noindent\textbf{Initial stage.} Suppose $Y_0 \ge c^{\ep}$ and let $t^\star = \inf\{t:Y_t \le c^{\ep}\}$. If $Y_0 \ge c^\ep$ then $Y_0 \to \infty$ so $g(Y_0)=2\log c - \log r+o(1)$. Since the result is proved for $Y_0\le c^{\ep}$ it is enough to show that $t^\star = o(1)$. From the drift of $Y$ and Jensen's inequality applied to $\E[Y_t^2]$ we find that $u(t) = \E[Y_t]$ satisfies the differential inequality
$$u' \le -u - ru^2 \le -ru^2.$$
Integrating the inequality,
$$u(t) \le ( u(0)^{-1} + rt)^{-1} \le 1/rt.$$
By Markov's inequality, $Y_t \le c^{\ep}$ or equivalently $t^\star \le t$ with high probability if $u(t) = o(c^{\ep})$, which is the case if $t=c^{-\ep/2}$. Since $c^{-\ep/2}=o(1)$, the result is proved.
\end{proof}

\section{Non-linear diffusive}\label{sec:crit}

\begin{proof}[Proof of Theorem \ref{thm:nonlin-diff}]
The first part of the proof is to show convergence to the limiting diffusion, which is done in \cite{crit-scale} by convergence of generators - here we do it using a general result, Lemma \ref{lem:limproc}, that requires convergence of drift and diffusion coefficients and vanishing jump size. The second part of the proof is to show the extinction time is short when the initial value is small, relative to the space and time scale of the limiting diffusion.\\

For this section let $Y_n(t)=X_n(\sqrt{n}t)/\sqrt{n}$ and let $c_n=\sqrt{n}(r_n-1)$ as in the statement of the theorem. There are two cases to cover: $Y_n(0) \to y \in (0,\infty)$ and $Y_n(0) \to \infty$.\\

\noindent\textbf{Case 1: $Y_n(0) \to y\in (0,\infty)$.} First we use Lemma \ref{lem:limproc} in the Appendix, which is a result from \cite{EthierKurtz}, to show that if $Y_n(0) \to y$ then for all but countably many $R$, $Y_n(\cdot\wedge \tau_n^R) \cd Y(\cdot\wedge \tau^R)$, where $\tau_n^R$ and $\tau^R$ are the exit times of $Y_n$ and $Y$ from $(1/R,R)$ as described in Lemma \ref{lem:limproc}. Recall that $Y_n(t)$ has transitions
$$Y_n \to \begin{cases}Y_n+1/\sqrt{n} & \text{at rate} \quad nY + \sqrt{n}Y(c_n-Y) - c_nY, \\
         Y_n-1/\sqrt{n} & \text{at rate} \quad nY,
        \end{cases}$$
so $Y_n$ has jump size $1/\sqrt{n}=o(1)$ and
$$\mu(Y_n) = Y(c_n-Y) - c_nY/\sqrt{n} \quad \text{and} \quad \sigma^2(Y_n) = 2Y + Y(c_n-Y)/\sqrt{n} - c_nY/n.$$
For $|y| \le R$, $\mu(y) \to b(y)$ and $\sigma^2(y) \to a(y)$ uniformly, where
$$b(y) = y(c_\infty-y) \quad \text{and} \quad a(y) = 2y.$$
Note that $b$ and $\sqrt{a}$ are Lipschitz on compact subsets of $(0,\infty)$. By Lemma \ref{lem:limproc}, the desired convergence holds.\\

For $y>0$ define the mapping $T_y(f) = \inf\{t:f(t) \le y\}$ from c{\`a}dl{\`a}g functions $f:\R_+\to \R_+$ with the topology of uniform convergence on compacts. If $f_i \to f$ and $f$ is continuous then since $\inf\{f(t):t \in [0,T_y(f)-\ep]\}>y$ for any $\ep>0$, it follows that
$\liminf_i T_y(f_i) \ge T_y(f)$. On the other hand, if $y>0$ then for any $\ep>0$, $Y(T(y)+\ep)$ intersects $[0,y)$, since its diffusion coefficient at $y$ is non-zero. In other words, $\Q(f\colon \ \text{if} \ f_i \to f \ \text{then} \ \limsup_i T_y(f_i) \le T_y(f))=1$, where $\Q$ is the law of $Y$. Combining the two, the discontinuity points of $T_y$ have $\Q$-measure $0$. Let $T_n(y) = T_y(Y_n)$ and $T(y) = T_y(Y)$, and let $T_n=T_n(0)$ and $T=T(0)$. 
By the continuous mapping theorem and convergence of $Y_n$ to $Y$,
$$T_n(y) \cd T(y) \quad \text{for} \quad y >0.$$
Let $\P_n$ be the law of $Y_n$. Since $T_n \ge \lim_{y\to 0^+}T_n(y)$ and $T = \lim_{y\to 0^+}T(y)$,
$$\limsup_n \P_n(T_n \le t) \le \limsup_{y\to 0^+}\lim_n \P_n(T_n(y) \le t) = \lim_{y\to 0^+} \Q(T(y) \le t) = \Q(T \le t).$$
To obtain the opposite inequality it's enough to show that for any $\ep,t>0$ there are $\alpha,n_0$ so that \\ $\P_n(T_n >t \mid Y_n(0)\le \alpha)\le \ep$ for $n\ge n_0$, since then
\begin{align*}
\liminf_n \P_n(T_n \le t) & \ge \liminf_n \P_n(T_n(\alpha) \le t)\inf_{x \le \alpha}\P_n(T_n \le t \mid Y_n(0) =x) \\
& \ge (1-\ep)\lim_{n\to\infty}\P_n(T_n(\alpha) \le t) \\
& = (1-\ep)\Q(T(\alpha) \le t) \ge (1-\ep)\Q(T \le t).
\end{align*}
Returning to the original time scale, $X_n$ is dominated by the linear birth-and-death process $Z$ with parameter $r$ and $Z_0=X_n(0)$, so using $\rho(t)$ as in the proof of Theorem \ref{thm:bbp},
$$\P(X_n(\sqrt{n} t)>0 \mid X_n(0) \le \sqrt{n}\alpha ) \le 1-(\rho(\sqrt{n}t))^{\sqrt{n}\alpha}.$$
If $c_n\le 0$ we can take $r=1$ which has $\rho(t) = 1-1/(1+t)$, so $1-(\rho(\sqrt{n}t))^{\sqrt{n}\alpha} \le \alpha/t$ is at most $\ep$ if $\alpha \le \ep/t$.\\
If $c_n>0$ re-write $\rho(t)$ from \eqref{eq:rhot} with $\dlt=r-1$ instead of $\gma = 1-r$ to obtain
$$\rho(t)^{-1} = \frac{r e^{\dlt t}-1}{e^{\dlt t}-1} = 1 + \frac{\dlt}{1-e^{-\dlt t}}.$$
Using $\dlt_n=c_n/\sqrt{n}$,
$$(\rho(\sqrt{n}t))^{\sqrt{n}\alpha} = (1 + \frac{c_n/\sqrt{n}}{1-e^{-c_nt}})^{-\sqrt{n}\alpha} \to \exp(-c_\infty\alpha/(1-e^{-c_\infty t})) \quad \text{as} \quad n\to\infty,$$
and since for fixed $t$ the limit $\to 1$ uniformly as $\alpha \to 0$, we are done.\\

\noindent\textbf{Case 2: $Y_n(0) \to \infty$.} It remains to show the results of Step 2 are true for $y=\infty$. First we need to make sense of $T$ when $Y(0)=\infty$. Let $T(y,w) = \inf\{t:Y(t) \le w \mid Y(0)=y\}$. Since $Y$ is continuous, and using the strong Markov property,
\begin{equation}\label{Tincr}
T(y,0) \stackrel{(d)}{=} T(y,w) + T(w,0),
\end{equation}
where the last two are independent. In particular, $T(y,0)$ dominates $T(w,0)$ for $y>w$. On the other hand, letting $U=Y-c_\infty$, $\mu(U)= -U(U+c_\infty) \le -U^2$ so integrating and using Jensen's inequality,
$$\E[U(t)] \le (1/\E[U_0] + t)^{-1} \le 1/t.$$
Using Markov's inequality, if $y>w>c_\infty$ then
$$\P(T(y,w) > t) \le \P(U(t) > w-c_\infty \mid Y(0) = y) \le ((w-c_\infty)t)^{-1}.$$
It follows that $T(y,w) \cp 0$ uniformly over $y \in [w,\infty)$ as $w\to\infty$. Combining with \eqref{Tincr}, there exists $T(\infty,0)$ such that $T(y,0) \cd T(\infty,0)$ as $y\to\infty$. A similar argument shows that for $T_n(y,0) = \inf\{t:Y_n(t)=0 \mid Y_n(0) = y\}$ there is a limit $T_n(\infty,0)$. By Step 2, $T_n(y,0) \cd T(y,0)$ for each $y>0$, so it follows that $T_n(\infty,0) \cd T(\infty,0)$.
\end{proof}

\section{Threshold and metastable}\label{sec:super}

In this section we prove Theorem \ref{thm:meta}. As we've done so far in the paper, we'll avoid writing subscript $n$ on everything. So, for example $X_n^\star$ is simply denoted $X_\star$, etc. Since we'll need some additional decorations later, we'll move the $\star$ from $\tau_n^\star$ into the subscript, so, $\tau_\star = \inf\{t \colon X_t \in \{0,X_\star\}\}$. We begin with some basic theory and estimation of an important function. Let $q_+(j)=rj(1-j/n)$ and $q_-(j)=j$ denote the transition rates of $X$ and for later use, let $q(j)=q_+(j)+q_-(j)$. For integer $j$, define
\begin{align*}
h_+(j) &= \P(X_{\tau_\star}=X_\star \mid X_0=j) \quad \text{and} \\
\quad h_-(j) &= \P(X_{\tau_\star}=0 \mid X_0=j) = 1-h_+(j).
\end{align*}

By definition, $h_+(X(t\wedge \tau_\star))$ is a martingale, as is $h_-(X(t\wedge \tau_\star))$. Using the generator of the process, it follows that $q_+(j)(h_+(j+1)-h_+(j))+q_-(j)(h_+(j-1)-h_+(j))=0$ and similarly for $h_-$. Let $\nu(0)=1$ and for $j \ge 1$ let $\nu(j)= \prod_{i=1}^j q_-(i)/q_+(i)$. Using the linear equations for $h_+,h_-$ and the boundary conditions $h_+(0)=0$, $h_+(X_\star)=1$, $h_-(0)=1$, $h_-(X_\star)=0$, we can solve to find that
\begin{align}\label{eq:hrmc}
h_+(j) = \frac{\sum_{k=0}^{j-1}\nu(k)}{\sum_{k=0}^{X_\star-1}\nu(k)}\quad \text{and} \quad 
h_-(j) = \frac{\sum_{k=j}^{X_\star-1} \nu(k)}{\sum_{k=0}^{X_\star-1}\nu(k)}.
\end{align}
The solution of the above linear equations for $h_+,h_-$ to obtain \eqref{eq:hrmc} is not hard; it can be found, for example, in Example 5.3.9 of \cite{PTE}. We begin by estimating $\nu(k)$. 
Since it not more difficult to estimate, and since we'll need it later, we'll estimate the more general $\smash{\nu(j,k) = \prod_{i=j+1}^k q_-(i)/q_+(i)}$, defined for $0\le j<k <n$; we recover from it $\nu(k)= \nu(0,k)$. It will be helpful to have both a general upper bound and a precise estimate.\\

\begin{lemma}\label{lem:nu}
Let $V(x) = x(\log r - 1) - (1-x)\log(1-x)$. Then for integer $na,nb$ with $0 \le a < b \le 1$,
$$\nu(na,nb) \le \exp(-n(V(b+1/n)-V(a+1/n)))$$
and
\begin{align*}
&\nu(na,nb) = \sqrt{\frac{1-a}{1-b}}\exp(-n(V(b)-V(a))E_n(a,b), \quad \text{with} \\
&|\log E_n(a,b)| \le (12n(1-b)^2(b-a))^{-1}.
\end{align*}
\end{lemma}

\begin{proof}
Since $q_+(i)/q_-(i) = r(1-i/n)$,
\begin{align}\label{eq:nujk}
-\log \nu(j,k) = \sum_{i=j+1}^k \left(\log r + \log(1-i/n) \right) = (k-j)\log r - \sum_{i=j+1}^k f(i/n),
\end{align}
where $f(x)=-\log(1-x)$ is positive and increasing for $x \in (0,1)$. Since $f(x)$ has the antiderivative $x + (1-x)\log (1-x)$, the upper bound follows. Using a trapezoidal approximation with $k-j$ subintervals of size $1/n$ and writing the approximation as an upper Riemann sum minus a telescoping triangular correction,
\begin{align}\label{eq:trap-nu}
\int_{j/n}^{k/n} f(x)dx = \frac{1}{n}\sum_{i=j+1}^{k} f(i/n) - \frac{1}{2n}(f(k/n)-f(j/n)) + R_n(j,k),
\end{align}
where the error term (see \cite{trap-rule} for a simple proof) has the bound
$$|R_n(j,k)| \le \frac{\max_{x \in [j/n,k/n]}|f''(x)|}{12(k/n-j/n)n^2} = \frac{1}{(1-k/n)^2}\frac{1}{12(k/n-j/n)n^2}.$$
and using the antiderivative of $f$ together with \eqref{eq:nujk} and \eqref{eq:trap-nu},
$$-\frac{1}{n}\log \nu(na,nb) = V(b)-V(a)+\frac{1}{2n}\log\frac{1-b}{1-a}+R_n(na,nb)$$
and the precise estimate follows.
\end{proof}

Note that $V(x)$ has $V(0) = 0$, $V'(x) = \log(r(1-x))$ and $V''(x) = -1/(1-x)$. In particular, it is concave on $[0,1)$ and has $V'(x_\star)=\log 1 = 0$, so is increasing and positive on $(0,x_\star)$ and decreasing on $(x_\star,1)$, with maximum $V_\star = V(x_\star) = \log r + 1/r - 1>0$, and has $V''(x_\star) = -1/(1-(1-1/r)) = -r$. If $\dlt_\infty = \lim_n \dlt>0$ then $V_\star$ has a positive limit, while if $\dlt \to 0$ then $V_\star = \log(1+\dlt)+1/(1+\dlt)-1 = \dlt - \dlt^2/2 - \dlt + \dlt^2 + O(\dlt^3) \sim \dlt^2/2$.

\subsection{Extinction probability}\label{sec:ext-prob}

In this section we prove the estimate of extinction probability in Theorem \ref{thm:meta}. The result in the case $\dlt_\infty>0$ belongs to (A2) in Theorem 1 of \cite{thresh-lim}, while in the case $\dlt \to 0$ it is new. The approach is to estimate $h_-(X_0)$ for $h_-$ as in \eqref{eq:hrmc}, and the proof works by estimating the values of $\nu(k)$ separately for small and large $k$.
\begin{lemma}\label{lem:meta-thr}
Suppose $\dlt X_0 \to a_\infty \in (0,\infty)$. Then,
$$\P(X_{\tau_\star}=0 ) \to \begin{cases} e^{-a_\infty} & \text{if} \quad \dlt \to 0 \\
(1+\dlt_\infty)^{-X_0} & \text{if} \quad \dlt \to \dlt_\infty>0.\end{cases}$$
If $\dlt X_0 \to \infty$ then $\P(X_{\tau_\star}=X_\star ) \to 1$.
\end{lemma}

\begin{proof}
The quantity of interest is
$$\P(X_{\tau_\star}=0 ) = h_-(X_0) = \frac{\sum_{k=X_0}^{X_\star-1}\nu(k)}{\sum_{k=0}^{X_\star-1}\nu(k)}.$$
Since $q_-(k)/q_+(k) = (r(1-k/n))^{-1}$, for any $M \ge 1$ and $k \in \{1,\dots,M\}$,
$$r^{-k} \le \nu(k) \le (r(1-M/n))^{-k}$$
and so
\begin{align}\label{eq:nu-bnd}
r^{-j}\frac{1-r^{-(M-j)}}{1-1/r} \le \sum_{k=j}^{M-1}\nu(k) \le r^{-j}\frac{1}{1-(r(1-M/n))^{-1}}.
\end{align}
\textit{Case 1: $\dlt\to 0$.} Suppose $M$ is taken large enough that $\dlt(M-j) \to\infty$, and small enough that $M = o(\dlt n)$. 
Then $r^{-(M-j)} = (1+\dlt)^{-(M-j)} \le e^{-\dlt(M-j)} \to 0$ and $r(1-M/n) = (1+\dlt)(1-M/n) = 1 + \dlt + o(\dlt)$, so that
$$1-1/(r(1-M/n)) = 1-1/(1+\dlt+o(\dlt)) = 1-(1-\dlt+o(\dlt)) \sim \dlt.$$
Similarly, $1-1/r \sim \dlt$. Since $\dlt\to 0$, $r^{-j} = (1+\dlt)^{-j} \sim e^{-\dlt j}$, so
\begin{align}\label{eq:nu-bnd2}
\sum_{k=j}^{M-1}\nu(k) \sim e^{-\dlt j}/\dlt.
\end{align}
We have $V'(0) = \log r = \dlt-o(\dlt)$, and $x_\star = \dlt/r=o(1)$, so $V(x) \sim x$ uniformly over $x \in [0,x_\star]$, as $n\to\infty$. Using the upper bound from Lemma \ref{lem:nu} with $a=0$ and $nb$ in the range $\{M,\dots,X_\star-1\}$, for large $n$
$$\nu(nb) = \nu(0,nb) \le \exp(-\dlt nb/2).$$
Summing over $nb$,
$$\sum_{k=M}^{X_\star-1}\nu(k) \le \frac{e^{-\dlt M/2}}{1-e^{-\dlt/2}} = o(1/\dlt),$$
since $\dlt M\to\infty$ by assumption and $1-e^{-\dlt/2} \sim \dlt/2$. Noting that $1-1/r \sim \dlt$ and combining with \eqref{eq:nu-bnd2}, 
\begin{align}\label{eq:nu-sum0}
\sum_{k=j}^{X_\star-1}\nu(k) \sim e^{-\dlt j}/\dlt.
\end{align}
Using the values $j=0$ and $j=X_0$, we conclude that if $\dlt M - \dlt X_0 \to\infty$ (which also implies $\dlt M \to\infty$) and $M=o(\dlt n)$, then
$$h_-(X_0) \sim \frac{e^{-\dlt X_0}/\dlt}{e^{-0}/\dlt} = e^{-\dlt X_0} \to e^{-a_\infty}.$$
If $a_\infty<\infty$, since $\dlt X_0$ has a finite limit and $\sqrt{n}\dlt\to\infty$ it is easy to check that $M=\sqrt{n}(\sqrt{n}\dlt)^{1/2}$ satisfies the conditions. If $a_\infty=\infty$, since $e^{-a_\infty}=0$ and $h_-(X_0)$ decreases with $X_0$ it is enough to consider the case where $\dlt X_0\to\infty$ arbitrarily slowly; thus, to satisfy the condition $\dlt M - \dlt X_0 \to\infty$ it is sufficient that $\dlt M \to \infty$, for which the above choice of $M$ suffices.\\

\noindent\textit{Case 2: $\dlt \to \dlt_\infty >0$.} Note that $r\to r_\infty = (1+\dlt_\infty)>1$. Also, the condition $\dlt X_0\to a_\infty$ is equivalent to $X_0$ eventually constant if $a_\infty<\infty$, and to $X_0\to\infty$ if $a_\infty=\infty$. Suppose that $(M-j)\to\infty$ and $M=o(n)$. Then, $r^{-(M-j)} \to 0$ and $1-M/n \to 1$. From \eqref{eq:nu-bnd} we find
\begin{align}\label{eq:nu-bnd3}
\sum_{k=j}^{M-1}\nu(k) \sim r^{-j}/(1-1/r_\infty).
\end{align}
Since $V(0)=0$, $V$ is concave and both $x_\star$ and $V(x_\star)$ have a positive limit, for some constant $C_1>0$, eventually $V(x) \ge C_1x$ for $x \in [0,x_\star]$. Using Lemma 2 as before with $a=0$ and $nb$ in the range $\{M,\dots,X_\star-1\}$, since $\dlt$ is bounded by assumption, $1-b \ge 1-x_\star$ which has a positive lower bound and $nb \ge nM \to\infty$, so uniformly $|\log E_n(a,b)| \to 0$ and for large $n$ and some constant $C_2>0$,
$$\nu(nb) = \nu(0,nb) \le C_2e^{-C_1 nb}.$$
Again, summing over $nb$, 
$$\sum_{k=M}^{X_\star-1}\nu(k) \le C_2\frac{e^{-C_1 M}}{1-e^{-C_1}} \to 0.$$
Using the values $j=0$ and $j=X_0$ for constant $X_0<\infty$ and letting $M=\sqrt{n}$, combining the above with \eqref{eq:nu-bnd3} we find
$$h_-(X_0) = \frac{r^{-X_0}/(1-1/r_\infty) + o(1)}{1-1/r_\infty + o(1)} \to r_\infty^{-X_0}.$$
If $X_0\to\infty$ we may again assume it does so arbitrarily slowly, in which case $M=\sqrt{n}$ again suffices. Using $j=0$ and $j=X_0$ as above we find the numerator $\to 0$ while the denominator $\to 1-1/r_\infty>0$, so $h_-(X_0) \to 0$ as desired.
\end{proof}

\subsection{Rapid extinction}

Next we prove the results on rapid extinction from Theorem \ref{thm:meta}. Define the probability measures $P^\star$ and $P^0$ for events $E$ by
\begin{align}\label{eq:cond-meas}
P^\star(E) = \P(E \mid X_{\tau_\star}=X_\star) \quad \text{and} \quad P^0(E) = \P(E \mid X_{\tau_\star} = 0).
\end{align}
Using the well-known Doob h-transform (which can be found by computing the generator of the conditioned process), we find that with respect to $P^\star$ and $P^0$ respectively, for $t<\tau_\star$, $X$ is a continuous-time Markov chain with transition rates
\begin{align}\label{eq:cond-trans}
q_+^\star(j) &= q_+(j)\frac{h_+(j+1)}{h_+(j)} \quad \text{and} \quad q_-^\star(j) = q_-(j) \frac{h_+(j-1)}{h_+(j)}, \quad \text{and} \nonumber \\
q_+^0(j) &= q_+(j)\frac{h_-(j+1)}{h_-(j)} \quad \text{and} \quad q_-^0(j) = q_-(j) \frac{h_-(j-1)}{h_-(j)}. 
\end{align}

\nid The following is an equivalent formulation of the rapid extinction results of Theorem \ref{thm:meta}.

\begin{lemma}\label{lem:super-die}
Suppose $\dlt X_0 \to a_\infty \in (0,\infty)$ and let $P^0$ be as in \eqref{eq:cond-meas}.
\begin{itemize}[noitemsep]
\item If $\dlt \to 0$, then for $w\ge 0$, $P^0( \tau \le w X_0) \to \exp(-a_\infty/(e^{-a_\infty w}-1))$.
\item If $\dlt \to \dlt_\infty >0$ then letting $\gma_\infty=\dlt_\infty/r_\infty$,
$$P^0(\tau \le t/r_\infty) \to (1+ \gma_\infty/(e^{-\gma_\infty t}-1))^{-X_0}.$$
\end{itemize}
\end{lemma}

\begin{proof}
We first estimate the transition rates \eqref{eq:cond-trans} with respect to $P^0$, then approximate by a linear birth-and-death process as in the proof of Theorems \ref{thm:disc} and \ref{thm:lin-diff} to obtain the scaling of $\tau$.
\begin{align}\label{eq:sdie-hp}
\frac{h_-(j+1)}{h_-(j)} = \frac{\sum_{k=j+1}^{X_\star-1}\nu(k)}{\sum_{k=j}^{X_\star-1}\nu(k)} = 1 - \frac{\nu(j)}{\sum_{k=j}^{X_\star-1}\nu(k)} = 1 - \frac{1}{\sum_{k=j}^{X_\star-1}\nu(j,k)}
\end{align}
and similarly,
\begin{align}\label{eq:sdie-hm}
\frac{h_-(j-1)}{h_-(j)} = 1 + \frac{1}{\sum_{k=j}^{X_\star-1}\nu(j-1,k)}.
\end{align}
Again we divide by cases.\\

\noindent\textit{Case 1: $\dlt\to 0$.} Let $c=\sqrt{n}\dlt$ so that $c\to\infty$ and $c=o(\sqrt{n})$ by assumption and let $M = \lf 2\sqrt{n/c} \rf$, so that $M/\sqrt{n} =o(1)$, and let $M' = \lf \sqrt{n} \rf$. For $j \le M$, write
\begin{align}\label{eq:sdie-sum1}
\sum_{k=j}^{X_\star-1}\nu(j,k) = \sum_{k=j}^{M'-1} \nu(j,k) + \sum_{k= M'}^{X_\star-1}\nu(j,k).
\end{align}
If $j \le k < M'$ then $r^{-(k-j)} \le \nu(j,k)\le (r(1-M'/n))^{-(k-j)}$, so uniformly over $j \le M$,
\begin{align}\label{eq:sdie-sum2}
\frac{1-r^{-(M'-M)}}{1-1/r} \le \sum_{k=j}^{M'-1} \nu(j,k) \le \frac{1}{1-1/(r(1-M'/n))}.
\end{align}
Thus the above sum $\sim 1/(1-1/r) \sim 1/\dlt$ uniformly over $j \le M$ provided $r^{-(M'-M)} \to 0$ and $M'/n = o(\dlt)$. The second point is clear, since $M'/n \le 1/\sqrt{n} =o(\dlt)$. To check the first point, since $r^{-(M'-M)}=(1+\dlt)^{-(M'-M)} \le e^{-\dlt(M'-M)}$ it is enough that $\dlt(M'-M)\to\infty$. Since $c=\sqrt{n}\dlt \to \infty$ and $\dlt=O(1)$,
$$\dlt(M'-M) = c - 2\sqrt{c} + O(\dlt) \to \infty.$$
To estimate the second sum on the RHS of \eqref{eq:sdie-sum1}, we note that since $V'(0)=\dlt-o(\dlt)$ and $x_\star=o(1)$, for large $n$, $V(x)-V(y)\ge \dlt(x-y)/2$ if $0 \le y<x\le x_\star$. With $na=j \le M$ and $M' \le nb \le X_\star-1$, $1-a,1-b \ge 1-x_\star\to 1$ and $nb-na \ge M'-M \to \infty$, so $\log|E_n(a,b)| \to 0$ uniformly over $a$ and $b$, thus for large $n$ and $j\le M$, $M' \le k \le X_\star-1$,
$$\nu(j,k) \le 2\exp(-\dlt(k-j)/2).$$
Summing over $k$, the second term in \eqref{eq:sdie-sum1} is $\le 2e^{-\dlt(M'-M)/2}/(1-e^{-\dlt/2}) \sim 4e^{-\dlt(M'-M)/2}/\dlt=o(1/\dlt)$. Combining the two estimates, it follows that the sum on the LHS of \eqref{eq:sdie-sum1} $\sim 1/\dlt$ uniformly over $j\le M$.\\

Since $\nu(j-1,k)=\nu(j-1,j)\nu(j,k)$ and $1/r \le \nu(j-1,j) \le 1/(r(1-M/n))$ which $\to 1$ uniformly over $j\le M$, using \eqref{eq:sdie-hp},\eqref{eq:sdie-hm} and \eqref{eq:cond-trans} we find that uniformly over $j \le M$,
\begin{align}\label{eq:sdie-crates}
q_+^0(j) \sim r \, j \, (1-j/n) \, (1 - \dlt) \quad \text{and} \quad q_-^0(j) \sim j \, (1+\dlt)
\end{align}
Let $\tld X$ denote the process with $\tld X_0=X_0$ and transition rates $\tld q_-(j) = j$ and $\tld q_+(j) = q_+^0(j)( j / q_-(j))$, which are the same as for $X$ w.r.t.~the measure $P^0$ except multiplied by the factor $j/q_-^0(j)$ at each non-zero $j$. Then $X$ w.r.t.~$P^0$ is obtained from $\tld X$ as $X_t = \tld X_{s(t)}$ for $t\le \tau$, where $s(t)$ is the inverse of the function 
$$t(s) = \int_0^s \frac{\tld X_u}{q_-^0(\tld X_u)}du.$$
From the estimate of \eqref{eq:sdie-crates}, we have
$$\sup_{t \le \tau_M}s(t)/t \sim 1+\dlt \to 1,$$
so to obtain the desired result for $X$ w.r.t.~$P^0$ it is enough to show it for $\tld X$. Since $M=o(\sqrt{n})$, $M/\dlt n = o(1/\sqrt{n}\dlt) = o(1)$. Using $r=1+\dlt$, uniformly over $j\le M$ we have
$$\tld q_+(j)/j \sim (1+\dlt)(1-o(\dlt))(1-\dlt)/(1+\dlt) = 1-\dlt+o(\dlt).$$
Let $\gma,\gma'$ be lower and upper bounds on $1-\tld q_+(j)/j$, respectively, and construct bbps $Z,Z'$ with parameters $1-\gma$ and $1-\gma'$, and initial value $X_0$, so that $Z'_t \le X_t \le Z_t$ for $t\le \tau_M$. Following the proof of Theorem \ref{thm:lin-diff}, combining \eqref{eq:X-sm} and \eqref{eq:taum-1} we have
$$\P(\tau \ne \tau_M \mid X_0) \le (1-\gma)^{M-X_0},$$
which $\to 0$ if $X_0 \le \sqrt{n/c}$ since then $\gma \sim \dlt$ and $\dlt (M-X_0) \ge \dlt \sqrt{n/c} = \sqrt{c} \to \infty$. As in the proof of Theorem \ref{thm:lin-diff}, it only remains to check that $(\gma-\dlt)X_0,(\gma'-\dlt)X_0 \to 0$. This follows easily from the fact that $\gma,\gma' \sim \dlt$.\\

\noindent\textit{Case 2: $\dlt \to \dlt_\infty >0$.} Let $r_\infty=1+\dlt_\infty$. Since $\dlt X_0$ converges, we may assume $X_0$ is constant. We follow the same approach as before, only with different $M$. So, let $M'=\lf \sqrt{n} \rf$ and let $M\to \infty$ slowly. Then $M'/n \to 0$ and $M'-M\to\infty$ and since $r \to r_\infty>1$, $r^{-(M'-M)} \to 0$. Thus from \eqref{eq:sdie-sum2}, the first sum on the RHS of \eqref{eq:sdie-sum1} $\sim 1/(1-1/r) \to r_\infty/\dlt_\infty$. To estimate the second sum, note that if $j\le M$ and $M' \le k \le X_\star-1$ then
$$\nu(j,k) = \nu(j,M')\nu(M',X_\star-1) \le \nu(M,M') \le (r(1-M'/n))^{-(M'-M)}.$$
We may assume $M=o(\sqrt{n})$, then for large $n$, the RHS above is $\le e^{-\dlt_\infty\sqrt{n}/2}$. Summing over at most $n$ such terms, the second sum on the RHS of \eqref{eq:sdie-sum1} is $\le ne^{-\dlt_\infty \sqrt{n}/2} = o(1)$, so combining the two, the LHS of \eqref{eq:sdie-sum1} $\to r_\infty/\dlt_\infty$ uniformly for $j\le M$.\\

Recall that $\nu(j-1,k)=\nu(j-1,j)\nu(j,k)$. Uniformly for $j\le M$, $\nu(j-1,j)\to 1/r_\infty$ and $(1-j/n) \to 1$. In addition, $1-\dlt_\infty/r_\infty = 1/r_\infty$. Thus uniformly over $j\le M$,
$$q_+^0(j) \sim r_\infty j(1-\dlt_\infty/r_\infty) = j \quad \text{and} \quad q_-^0(j) \sim j(1+\dlt_\infty) = r_\infty j.$$
Let $\tld X$ be as before, which has $\tld q_-(j) = j$ and $\tld q_+(j) \sim j/r_\infty$. In this case, $X_t = \tld X_{s(t)}$ with $\sup_{t \le \tau_m}s(t)/t \sim r_\infty$. Thus the result for $X$ is obtained from the one for $\tld X$ by changing time by the factor $r_\infty$. For $\tld X$, compare to the linear birth-and-death process $Z$ with parameter $1/r_\infty$. Following the proof of Theorem \ref{thm:disc} it suffices to show that for fixed $t>0$, $\P(\tau \ne \tau_M),\P(\tld X_s \ne Z_s \ \text{for some} \ s \le t \wedge \tau_M)=o(1)$. Since $\tld q_+(j) \le \tld q_-(j)$, $\tld X$ is a supermartingale, so the first statement is true provided $M\to\infty$. For $t< \tau_m$ the difference in rates between $X$ and $Z$, when they take the same value, is at most $\sup_{j \le M}\tld |q_+(j) - j/r_\infty| = o(M)$, so is $o(1)$ if $M\to\infty$ slowly enough. Thus for constant $X_0$ and fixed $t>0$,
$$\P(X_s \ne Z_s \ \text{for some} \ s \le t \wedge \tau_M \mid X_0) \le 1-e^{-o(1)t} = o(1)$$
which shows the second statement and completes the proof.
\end{proof}

\subsection{Metastability}

In this section we prove the metastability results of Theorem \ref{thm:meta}. In particular, in this section we condition on $A_\star = \{X_{\tau_\star} = X_\star\}$, the event denoted $A_n^\star$ in Theorem \ref{thm:meta}. On $A_\star$, $\tau_\star$ is the time of the first visit to $X_\star$, and the time to extinction can be broken into three epochs. Define the time of the last visit to $X_\star$ as
$$\tau_\star^o = \sup\{t\colon X_t=X_\star\},$$
setting $\tau_\star^o=-\infty$ if $X$ never reaches $X_\star$. On $A_\star$, the time to extinction $\tau=\inf\{t \colon X_t=0\}$ is the sum of the \emph{approach time} $\tau_\star$, the \emph{sojourn time} $\tau_\star^o - \tau_\star$ and the \emph{fall time} $\tau-\tau_\star^o$.
We proceed as follows:
\begin{enumerate}[noitemsep]
\item Estimate the expected sojourn time $E_\star^o := \E[ \ \tau_\star^o-\tau_\star \mid X_{\tau_\star}=X_\star \ ]$.
\item Show the expected approach time and fall time are $o(E_\star^o)$.
\item With a coupling argument, show that the rescaled sojourn time $(\tau_\star^o-\tau_\star)/E_\star^o$, \\
conditioned on $X_{\tau_\star}=X_\star$, converges in distribution to exponential with mean 1.
\end{enumerate}

Let us make the formal statements that will be the goal of this section.

\begin{lemma}[Expected sojourn time]\label{lem:soj} $\displaystyle E_\star^o \sim \sqrt{\frac{2\pi}{n}}\frac{r}{\dlt^2}\exp(n(\log r + 1/r-1))$.
\end{lemma}

\begin{lemma}[Approach time]\label{lem:appr} $\max_{j \in \{1,\dots,n\}}\E[\tau_\star \mid X_{\tau_\star}=X_\star, \ X_0=j] = o(E_\star^o)$.
\end{lemma}

\begin{lemma}[Fall time]\label{lem:fall} $\E[\tau-\tau_\star^o] = o(E_\star^o)$.
\end{lemma}

\begin{lemma}[Exponential limit]\label{lem:exp}
For each $t\ge 0$, $\P((\tau_\star^o-\tau_\star)/E_\star^o > t) \to e^{-t}$.
\end{lemma}

Note that the expected sojourn time and fall time does not depend on $X_0$; we have emphasized the uniformity of the estimate with respect to $X_0$ only in Lemma \ref{lem:appr}. Before proving these results, let's use them to prove the rest of Theorem \ref{thm:meta}.

\begin{proof}[Proof of metastability results of Theorem \ref{thm:meta}]
The extinction time is the sum of the approach, sojourn, and fall time: $\tau = \tau_\star + (\tau_\star^o-\tau_\star) + (\tau-\tau\star^o)$. Combining Lemmas \ref{lem:soj}, \ref{lem:appr} and \ref{lem:fall}, $\E[\tau \mid X_{\tau_\star}=X_\star] \sim E_\star^o$ which is the desired estimate of expected time to extinction.\\

Using Lemmas \ref{lem:appr} and \ref{lem:fall} and Markov's inequality, conditioned on $X_{\tau_\star}=X_\star$, with probability $1-o(1)$, $\tau/E_\star^o = (\tau_\star^o-\tau_\star)/E_\star^o + o(1)$. The exponential limit for $\tau$ then follows from Lemma \ref{lem:exp}.
\end{proof}

We begin by deriving some formulas for the expected time to hit $j+1$ or $j-1$, starting from $j$.

\subsubsection{Crossing times}\label{sec:cross-time}
Let $T_+(j) = \inf\{t \colon X_t = j + 1\}$ and $T_-(j) = \inf\{t \colon X_t=j-1\}$, and let $S_+(j) = \E[T_+(j) \mid X_0 =j]$ and $S_-(j) = \E[T_-(j) \mid X_0=j]$. Define also the conditioned versions
\begin{align*}
S_\pm^\star(j) &= \E[T_\pm(j) \mid X_0=j \ \text{and} \ X_{\tau_\star}=X_\star], \\
S_\pm^0(j) &= \E[T_\pm(j) \mid X_0=j \ \text{and} \ X_{\tau_\star}=0].
\end{align*}
We will need to estimate the following quantities:
\begin{enumerate}[noitemsep,label={\roman*)}]
\item $S_+^\star(j)$ for $j<X_\star$,
\item $S_-(j)$ for $1 \le j \le n$, and
\item $S_-^0(j)$ for $j<X_\star$.

\end{enumerate}
\textit{Quantity i).} For $1<j<X_\star$ a first step analysis gives
$$S_+^\star(j) = \frac{1}{q_+^\star(j)+ q_-^\star(j)} + \frac{q_-^\star(j)}{q_+^\star(j) + q_-^\star(j)}(S_+^\star(j-1) + S_+^\star(j)),$$
and solving gives
$$q_+^\star(j)S_+^\star(j) = 1 + q_-^\star(j)S_+^\star(j-1).$$
Following \cite[Sec 5.2]{rar-exp} we let $\pi(i,i) = 1$ and $\pi(i,j) = \prod_{k=i}^{j-1} q_+^\star(k)/q_-^\star(k+1)$ and multiply through by $\pi(1,j)$ above to obtain
$$q_+^\star(j)\pi(1,j)S_+^\star(j) = \pi(1,j) + q_+^\star(j-1)\pi(1,j-1)S_+^\star(j-1).$$
and then, since $q_+^\star(1)S_1^+ = 1$, we solve to obtain
$$S_+^\star(j) = \frac{1}{q_+^\star(j)\pi(1,j)}\sum_{i=1}^j \pi(1,i) = \frac{1}{q_+^\star(j)}\sum_{i=1}^j \frac{1}{\pi(i,j)}.$$
Using \eqref{eq:cond-trans}, for $i < j$
\begin{align*}
\frac{1}{\pi(i,j)} &= \prod_{k=i}^{j-1} \frac{q_-(k+1)}{q_+(k)}\frac{h_+(k)/h_+(k+1)}{h_+(k+1)/h_+(k)} \\
&= \left(\frac{h_+(i)}{h_+(j)}\right)^2 \frac{q_-(j)}{q_+(i)} \ \nu(i,j-1).
\end{align*}
This gives
\begin{align}\label{eq:sj-up}
S_+^\star(j) = \frac{q_-(j)}{q_+^\star(j)}\sum_{i=1}^j \nu(i,j-1) \frac{h_+(i)^2}{h_+(j)^2} \frac{1}{q_+(i)}.
\end{align}
\textit{Quantity ii).} For $1\le j< n$, a first step analysis gives the recursion
$$q_-(j)S_-(j) = 1 + q_+(j)S_-(j+1).$$
Multiplying through by $1/\pi(j,n)$ gives
$$\frac{q_-(j)}{\pi(j,n)}S_-(j) = \frac{1}{\pi(j,n)} + \frac{q_-(j+1)}{\pi(j+1,n)}S_-(j+1)$$
and since $S_-(n) = 1/q_-(n)$, the solution is
$$S_-(j) = \frac{\pi(j,n)}{q_-(j)}\sum_{i=j}^n \frac{1}{\pi(i,n)} = \frac{1}{q_-(j)}\sum_{i=j}^n \pi(j,i).$$
Writing in terms of $\nu$ we obtain
\begin{align}\label{eq:sj-dn}
S_-(j) = \frac{q_+(j)}{q_-(j)}\sum_{i=j}^n \nu(i-1,j)\frac{1}{q_-(i)}
\end{align}
\textit{Quantity iii).} For $1 \le j < X_\star$ a first step analysis gives
$$q_-^0(j)S_-^0(j) = 1 + q_+^0(j)S_-^0(j+1).$$
In this case it is better to express the solution using the function $\nu^0$, defined like $\nu$ except with $q_\pm^0$ in place of $q_\pm$. Since $S_-^0(X_\star-1) = 1/q_-^0(X_\star-1)$, following the same approach as above we find that
\begin{align}\label{eq:sj0-dn}
S_-^0(j) = \frac{q_+^0(j)}{q_-^0(j)}\sum_{i=j}^{X_\star-1}\nu^0(i-1,j)\frac{1}{q_-^0(i)}.
\end{align}

\subsubsection{Sojourn time}

Here we prove Lemma \ref{lem:soj}. We use the same basic approach as in \cite{thresh-lim}; the calculations become more delicate when $\dlt\to 0$. Define recursively $\tau_\star(0)=\tau_\star$ and for $0<k<K= \min\{ k >0 \colon X_{\tau_\star(k)}=0\}$,
$$\tau_\star(k+1)=\inf\{t > \tau_\star(k) \colon X_t \in \{0,X_\star\} \ \text{and} \ X_s \ne X_\star \ \text{for some} \ \tau_\star(k)<s<t\}.$$ For $k=1,\dots,K$ let $\rho_k = \tau_\star(k)-\tau_\star(k-1)$. Then, conditioned on $X_{\tau_\star} = X_\star$,
\begin{align}\label{eq:hold-time}
\tau_\star^0 - \tau_\star = \sum_{k=1}^{K-1} \rho_k.
\end{align}
By the strong Markov property, $K$ is geometric with success probability $p_\star =\P(X_{\tau_\star}= 0 \mid X_0=X_\star)$ and defining $\tau_\star^+ = \inf\{t \colon X_t=X_\star \ \text{and} \ X_s \ne X_\star \ \text{for some} \ 0 < s < t \}$,
$$\E[\rho_k \mid K>k] = \E[\tau_\star^+ \mid X_0=X_{\tau_\star^+}=X_\star].$$
Let $L_\star$ denote the above expectation. Applying Wald's equation to \eqref{eq:hold-time},
\begin{align}\label{eq:E-ht}
E_\star^o = L_\star(1/p_\star - 1).
\end{align}
Recall $V(x) = x(\log r - 1) - (1-x)\log(1-x)$ defined in Lemma \ref{lem:nu}, and $V_\star := V(x_\star)=\log r + 1/r-1$. Lemma \ref{lem:soj} follows immediately from the following estimates on $p_\star$ and $L_\star$.
\begin{lemma}\label{lem:ext-prob}
$p_\star \to 0$ and $\displaystyle p_\star \sim \frac{\dlt}{2\sqrt{r}}\exp(-nV_\star)$. 
\end{lemma}

\begin{lemma}\label{lem:Lstst}
$\displaystyle L_\star \sim \frac{1}{\sqrt{n}\,\dlt}\,\sqrt{\frac{\pi r}{2}}$.
\end{lemma}

First we prove Lemma \ref{lem:ext-prob}. 

\begin{proof}[Proof of Lemma \ref{lem:ext-prob}]
First we show that if $p_\star \sim (\dlt/2\sqrt{r})\exp(-nV_\star)$ then $p_\star \to 0$, then we establish the estimate. Since $r\ge 1$, $V_\star \ge 0$. If $\dlt \to 0$ it follows that $p_\star \to 0$. If $\dlt \to \dlt_\infty>0$ then $V_\star \to V_\infty = \log r_\infty + 1/r_\infty -1>0$ so $\exp(-nV_\star) \to 0$ and again $p_\star \to 0$.\\

If the first jump of $X$ is to $X_\star+1$ then $X_{\tau_\star}=X_\star$. So, conditioning on the first jump,
$$p_\star = \frac{q_-(X_\star)}{q(X_\star)}h_-(X_\star -1).$$
By definition of $X_\star$, $q_+(X_\star)\sim q_-(X_\star)$, so if $X_0=X_\star$ then its first jump is to $X_\star-1$ with probability $1/2+o(1)$. Thus, in the notation of \eqref{eq:hrmc} it is enough to show that
$$h_-(X_\star-1) \sim \frac{\dlt}{\sqrt{r}}\exp(-nV_\star).$$
We have
$$h_-(X_\star-1) = \frac{\nu(X_\star-1)}{\sum_{k=0}^{X_\star-1} \nu(k)},$$
and we begin by estimating the numerator, using Lemma \ref{lem:nu} with $a=0$ and $b = (X_\star-1)/n$. Notice that $|b-x_\star| \le 2/n$ and that $x_\star = \dlt/r = \dlt/(1+\dlt)$, so the assumption $\limsup_n \dlt<\infty$ implies $\limsup_n b < 1$. Since $V'(x_\star)=0$, $b<x_\star$ and $x \mapsto |V''(x)|$ is increasing, it follows that $|V(b)-V(x_\star)| \le \frac{1}{2}V''(x_\star)(2/n)^2 = O(1/n^2)$. Since $n b \ge n(x_\star-2/n) = n \dlt/r-2 \to \infty$ and $\limsup_n 1-b >0$, it follows that $n(1-b)^2(b-a) \to\infty$. Since $1/r=1-x_\star$, $1-b = 1-x_\star+O(1/n) = (1 - x_\star)(1+o(1)) \sim 1/r$. Putting it together,
\begin{align*}
\nu(X_\star-1) \sim \sqrt{r}\exp(-nV(x_\star)) \\
\end{align*}
The denominator is estimated in the proof of Lemma \ref{lem:meta-thr}; in both cases
$$\sum_{k=0}^{X_\star-1} \nu(k) \sim \frac{1}{1-1/r} = r/\dlt$$
and the result follows.
\end{proof}

In order to estimate $L_\star$ we will need additional information about the function $\nu$. First, extend the domain of $\nu$ by defining $\nu(k,j)=1/\nu(j,k)$ for $0\le j< k<n$. An equivalent, unifying definition is given by the formula
$$\nu(j,k) = \frac{\prod_{i=j+1}^nq_-(i)/q_+(i)}{\prod_{i=k+1}^nq_-(i)/q_+(i)}.$$
Say that $f(n,\lambda) \sim g(n,\lambda)$ uniformly over $\lambda \in A$ if $\lim_{n\to\infty} \sup_{\lambda \in A} \left|\log \left( \frac{f(n,\lambda)}{g(n,\lambda)} \right)\right| = 0$.

\begin{lemma}\label{lem:nu2}
Uniformly over $|\sigma| \le n^{1/8}$,
$$\nu(X_\star-\sigma \sqrt{n},X_\star) \sim \exp(-\sigma^2r/2).$$
Moreover, for $1 \le \sigma \le n^{1/8}$,
$$\sum_{0 \le j \le n-1 \colon |j-X_\star| \ge \sigma \sqrt{n}}\nu(j,X_\star) \le (2+o(1))\exp(-\sigma^2r/2)\frac{\sqrt{n}}{\sigma r}.$$
\end{lemma}

\begin{proof}
Since $r(1-x_\star)=1$ and $x\mapsto \log(r(1-x))$ is differentiable at $x_\star$, it follows that uniformly over $j$ such that $|j-X_\star| \le n^{1/4}$, $\log(r(1-j/n)) = O(n^{-3/4})$. Summing at most $n^{1/4}$ terms, $\log \nu(j,X_\star) = O(1/\sqrt{n})$. This proves the first statement restricted to $|\sigma| \le n^{-1/4}$.\\

If $n^{1/4} < |j-X_\star| \le n^{5/8}$, then since $\limsup_n x_\star <1$, $1-j/n \sim 1-X_\star/n$. Since $|j-X_\star| > n^{1/4}\to \infty$, by Lemma \ref{lem:nu}, uniformly over such $j$, $E_n(j/n,X_\star/n) \to 1$ if $j<X_\star$ and $E_n(X_\star/n,j/n) \to 1$ if $j \ge X_\star$, so
\begin{equation}\label{eq:nu1}
\nu(j,X_\star) \sim \exp(-n(V(X_\star/n)-V(j/n));
\end{equation}
note the last expression is valid not only for $j<X_\star$ but also for $j\ge X_\star$ under the extended definition of $\nu$. Since $\limsup_n 1-x_\star>0$, $V'''(x)=1/(1-x)^2$ is bounded on $[0,x_\star +o(1))]$, and recall $V'(x_\star)=0$ and $V''(x_\star) = -r$. Thus, if $|\sigma| \le n^{1/8}$, using a second order Taylor approximation we find that
\begin{equation}\label{eq:V1}
V(x_\star + \sigma /\sqrt{n}) - V(x_\star) = -\sigma^2 r/2n + O(n^{-9/8}).
\end{equation}
In particular, $V(X_\star/n)-V(x_\star) = O(n^{-9/8})$. Combining this with \eqref{eq:nu1} and \eqref{eq:V1}, the first statement is proved for the remaining values of $|\sigma|$, namely, $(n^{-1/4},n^{1/8}]$.\\

Next, for $j<k<X_\star$, since $\log(r(1-x_\star))=0$ and $x\mapsto \log(r(1-x))$ is decreasing,
$$-\log \nu(j,k) = \sum_{i=j+1}^k \log(r(1-i/n)) \ge (k-j)\log(r(1-k/n))$$
and thus $\nu(j,k) \le (r(1-k/n))^{-(k-j)}$. Fix $\sigma$ with $1 \le \sigma \le n^{1/8}$ and observe that for any $j,k,\ell$, $\nu(j,\ell)=\nu(j,k)\nu(k,\ell)$. Using this property with $k=X_\star-\sqrt{n}\sigma$ and $\ell=X_\star$ then bounding the sum by a geometric series,
$$\sum_{j=0}^{X_\star-\sqrt{n}\sigma}\nu(j,X_\star) \le \frac{(1+o(1))\exp(-\sigma^2r/2)}{1-(r(1-(X_\star-\sqrt{n}\sigma)/n))^{-1}}.$$
Since $r(1-X_\star/n) = r(1-x_\star)+O(1/n) = 1 + O(1/n)$, the denominator is
$$1-(1 + O(1/n) + \sigma r / \sqrt{n})^{-1} = 1-(1-(1+o(1))\sigma r/\sqrt{n}) = (1+o(1))n^{-1/2}\sigma r,$$
so
$$\sum_{j=0}^{X_\star-\sqrt{n}\sigma-1}\nu(j,X_\star) \le (1+o(1))\exp(-\sigma^2r/2)\frac{\sqrt{n}}{\sigma r}.$$
From the other end, if $X_\star<k<j$ then since $\log(r(1-i/n)) \le \log(r(1-x_\star))=0$ if $i>X_\star$, we have
$$\log \nu(j,k) = -\log \nu(k,j) = \sum_{i=k+1}^j \log(r(1-i/n)) \le (j-k)\log(r(1-k/n))$$
and thus $\nu(j,k) \le (r(1-k/n))^{j-k}$. By an analogous argument we find that
$$\sum_{j=X_\star+\sqrt{n}\sigma}^{n-1}\nu(j,X_\star) \le \frac{(1+o(1))\exp(-\sigma^2r/2)}{1-r(1-(X_\star+\sqrt{n}\sigma)/n)} \le 
(1+o(1))\exp(-\sigma^2r/2)\frac{\sqrt{n}}{\sigma r}.$$
Combining the two estimates completes the proof.
\end{proof}

We are now ready to prove Lemma \ref{lem:Lstst}.

\begin{proof}[Proof of Lemma \ref{lem:Lstst}]
Let $S_+^\star(j)$, $S_-(j)$ be as in Section \ref{sec:cross-time}. Conditioning on the first step,
$$L_\star = \frac{1}{q_+(X_\star) + q_-(X_\star)}\left( 1 + q_-(X_\star) S_+^\star(X_\star-1) + q_+(X_\star)S_-(X_\star+1) \right).$$
We have $q_-(X_\star)=X_\star$ and since $r(1-X_\star/n) = 1 + O(1/n)$, $q_+(X_\star) \sim X_\star$, so
$$L_\star \sim \frac{1}{2}\left(1/X_\star +  S_+^\star(X_\star-1) + S_-(X_\star+1) \right).$$
Since $h_+(X_\star)=1$ and $h_+(X_\star-1) \to 1$, $q_+^\star(X_\star -1) \sim q_+(X_\star-1)$. By definition of $X_\star$, $q_+(X_\star-1) \sim q_-(X_\star-1)$. In addition, by Lemma \ref{lem:nu2} $\nu(i,X_\star)=\nu(i,X_\star-2)\nu(X_\star-2,X_\star) \sim \nu(i,X_\star-2)$ uniformly over $i$. Thus, using \eqref{eq:sj-up} with $j=X_\star-1$,
\begin{align}\label{eq:s-star1}
S_+^\star(X_\star-1) \sim \sum_{i=1}^{X_\star -1} \nu(i,X_\star)
\frac{h_+(i)^2}{q_+(i)}.
\end{align}
We will estimate the bulk of the sum in \eqref{eq:s-star1}, finding that it tends to a then show the rest of the sum, as well as $S_-(X_\star+1)$, are negligible in comparison. Recall the notation $c=\sqrt{n}\dlt$, noting that $c\to\infty$ by assumption. We then have $X_\star \sim n\dlt = \sqrt{n}c$.\\

Since $r(1-x_\star)=1$, $q_+(i) = ri(1-i/n) = i(1+r(x_\star-i/n))$, so $q_+(i) \sim X_\star$ uniformly over $|i-X_\star| \le C$ provided $C=o(X_\star)$ or equivalently $C=o(\sqrt{n}c)$. By Lemma \ref{lem:meta-thr}, $h_+(i) \to 1$ uniformly over $i\ge M$ provided $\dlt M \to \infty$. If $M = X_\star - C$ with $C=o(\sqrt{n}c)$ then $\dlt M \sim \dlt X_\star \sim c^2/r \to \infty$. Thus, if we define $\Sigma = \sqrt{c} \wedge n^{1/8}$, then since $\sqrt{nc} = o(\sqrt{n}c)$,
$$\sum_{i= \lf X_\star - \Sigma\sqrt{n} \rf +1}^{X_\star-1} \nu(i,X_\star)\frac{h_+(i)^2}{q_+(i)}\quad \sim \quad \frac{1}{X_\star} \sum_{i= \lf X_\star - \Sigma\sqrt{n} \rf +1 }^{X_\star-1} \nu(i,X_\star).$$
Using Lemma \ref{lem:nu2} and the fact that $\Sigma \to\infty$ and $\Sigma \le n^{1/8}$,
$$\sum_{i= \lf X_\star - \Sigma\sqrt{n} \rf +1}^{X_\star-1} \nu(i,X_\star) \sim \sqrt{n}\int_{-\infty}^0 e^{-\sigma^2r/2}d\sigma = \sqrt{\frac{n\pi}{2r}}.$$
Assuming the rest of the sum is negligible in comparison, since $X_\star \sim n\dlt/r$ we then have
\begin{align}\label{eq:S-scale}
S_+^\star(X_\star-1) \sim \frac{\sqrt{n}}{X_\star}\frac{\pi}{2r} = \frac{1}{\sqrt{n}\dlt}\sqrt{\frac{\pi r}{2}}.
\end{align}
So, we now show the rest of the sum in the brackets in \eqref{eq:s-star1} is $o(\sqrt{n})$; we may of course ignore the 1. Using the trivial estimate $q_+(i) \ge 1$ and $h_+(i) \le 1$ for $i\ge 1$ as well as $X_\star \le n\dlt$ and Lemma \ref{lem:nu2} with $\sigma=n^{1/8}$, if $\sqrt{c} \ge n^{1/8}$ then the rest of the sum is
$$\sum_{i=1}^{\lf X_\star - n^{5/8} \rf}\nu(i,X_\star-2) \frac{h_+(i)^2}{q_+(i)}X_\star \le (2+o(1))e^{-n^{1/4}r/2}\frac{n^{3/8}}{r} \, \frac{n\dlt}{r} =O(1)= o(\sqrt{n}).$$
If $\sqrt{c} < n^{1/8}$ we treat the remainder in two parts, beginning with $i \le 1/\dlt = \sqrt{n}/c$. $X$ is dominated by the bbp $Z$ with $Z_0=X_0$ and parameter $r$, and using \eqref{eq:bbp-rho}, since $0$ is absorbing for $Z$,
$$\P(Z_t=0 \ \text{for some} \ t>0 \mid Z_0) = \lim_{t\to\infty}\rho(t)^{Z_0} = \min(1,1/r)^{Z_0} = r^{-Z_0}$$
since $r>1$ by assumption in this section. If $Z_t = i$ for $i>0$, then with probability $p_i>0$, $Z$ hits $0$ before it returns to $i$. Thus $Z$ visits any $i>0$ at most geometric$(p_i)$ number of times, which is almost surely finite and implies that a.s., $\lim_{t\to\infty}Z_t \in \{0,\infty\}$. Therefore
$$\P(\lim_{t\to\infty}Z_t =\infty \mid Z_0) = 1 - r^{-Z_0} = 1-(1+\dlt)^{-Z_0} \le \dlt Z_0.$$
The last inequality follows from $(1+\dlt)^{-Z_0} \ge e^{-\dlt Z_0} \ge 1-\dlt Z_0$, where we used the estimate $1+u \le e^u$ for $u \in \R$ which follows from convexity of $u\mapsto e^u$. Since $Z$ dominates $X$,
\begin{align}\label{eq:hp-ub}
h_+(i) \le \P(\lim_{t\to\infty}Z_t =\infty \mid Z_0=i) \le \dlt i.
\end{align}
Since $h_+(i) \le 1$, $h_+(i)^2 \le h_+(i)$, and since $X_\star \le n\dlt$ and $q_+(i) \ge i$ for $i<X_\star$, $h_+(i)^2 X_\star /q_+(i) \le \dlt^2n = c^2$. Using Lemma \ref{lem:nu2} with $\sigma = (X_\star - \lf \sqrt{n}/c \rf)/\sqrt{n} \sim c-1/c = c-o(1)$,
$$\sum_{i=1}^{\lf \sqrt{n}/c \rf} \nu(i,X_\star)\frac{h_+(i)^2}{q_+(i)}X_\star \le (2+o(1))e^{-\sigma^2 r/2}\frac{\sqrt{n}}{\sigma r} c^2 \sim 2e^{-(c-o(1))^2r/2} \frac{c}{r} \sqrt{n} = o(\sqrt{n}),$$
since $c\to\infty$ and thus $(2c/r)e^{-(c-o(1))^2r/2} \to 0$ as $n\to\infty$. Using again the trivial estimate $h_+(i)\le 1$ and $q_+(i) \ge i$, and also $\nu(i,X_\star) \le \nu(X_\star-\sqrt{nc},X_\star) \sim e^{-c r/2}$ for $i \le X_\star-\sqrt{nc}$,
\begin{align*}
\sum_{i=\lf \sqrt{n}/c \rf +1}^{\lf X_\star-\sqrt{nc} \rf}\nu(i,X_\star)\frac{h_+(i)^2}{q_+(i)}X_\star 
& \le (1+o(1))e^{-cr/2}\sqrt{n}c\sum_{i=\lf \sqrt{n}/c \rf+1}^{\lf X_\star-\sqrt{nc} \rf}\frac{1}{i}, \\
& \sim \sqrt{n}ce^{-cr/2}(\log(\sqrt{n}c)-\log(\sqrt{n}/c)) \\
& \sim \sqrt{n}(2c\log(c)e^{-cr/2}) = o(\sqrt{n}).
\end{align*}
This completes the estimation of the sum from \eqref{eq:s-star1}.\\
For $S_-(X_\star+1)$, using \eqref{eq:sj-dn} with $j=X_\star+1$ and simplifying as before,
$$S_-(X_\star+1) \sim \sum_{i=X_\star+1}^n \nu(i-1,X_\star)\frac{1}{q_-(i)}.$$
Since this case is similar to the one before, we just give an outline. Breaking up the sum in the same way, the bulk of the sum is estimated in the same way as before and gives the same result. To bound the remainder, it suffices to note that $q_-(i)=i \ge X_\star$ for $i \ge X_\star$, then use Lemma \ref{lem:nu2} directly, noting that $2e^{-\Sigma^2 r/2}/(\Sigma r) \to 0$. Since $L_\star$ scales like the average of the two values, the result follows from \eqref{eq:S-scale}.
\end{proof}

\subsubsection{Approach time}

We now prove Lemma \ref{lem:appr}. For this and for Lemma \ref{lem:fall}, we first derive a more concrete lower bound on the expected sojourn time $E_\star^o$, using the formula $E_\star^o \sim \sqrt{2\pi/n}(r/\dlt^2)\exp(n V_\star)$ of Lemma \ref{lem:soj}, where $V_\star=\log r + 1/r-1$.
\enumalph
\item If $\dlt \to 0$ then
$$V_\star = \log(1+\dlt)+1/(1+\dlt)-1 = \dlt-\dlt^2/2 + 1-\dlt + \dlt^2-1+O(\dlt^3) = \dlt^2/2 + O(\dlt^3) \sim \dlt^2/2$$
and since $c=\sqrt{n}\dlt$ and $r\to 1$, $\sqrt{2\pi/n}(r/\dlt^2) = \sqrt{2\pi}(r/c \dlt)$ and $nV_\star \sim c^2/2$.\\
\item If $\dlt\to\dlt_\infty>0$ then $V_\star \to V_\infty = \log r_\infty + 1/r_\infty -1>0$.
\enumend 
Thus in either case,
\begin{align}\label{eq:Estar-est}
E_\star^o \ge \begin{cases} \displaystyle(1-o(1))\frac{\sqrt{2\pi}}{\dlt}\,\frac{1}{c} \exp((1-o(1))c^2/2) & \text{if} \quad \dlt \to 0, \vspace{10pt} \\
\sqrt{2\pi/n}(r_\infty/\dlt_\infty^2)e^{(1-o(1))V_\infty n} & \text{if} \quad \dlt \to \dlt_\infty>0.\end{cases}
\end{align}
Using \eqref{eq:Estar-est}, the following result implies Lemma \ref{lem:appr}, since $c^2\log c = o((1/c)\exp((1-o(1))c^2/2))$ in the case $\dlt \to 0$ and $n\log n = o(\exp((1-o(1))V_\infty n))$ in the case $\dlt \to \dlt_\infty>0$.

\begin{lemma}\label{lem:equil-time}
$$\max_{j \in \{1,\dots,n\}}\E[\tau_\star \mid X_{\tau_\star}=X_\star, \ X_0=j] = 
\begin{cases} O((1/\dlt)c^2\log(c)) & \text{if} \ \dlt \to 0 \\
O(n\log(n)) & \text{if} \ \dlt \to \dlt_\infty>0.
\end{cases}$$
\end{lemma}

\begin{proof}
Using the natural coupling, by Lemma \ref{lem:nat-coup} it is enough to consider the initial values $X_0=1$ and $X_0=n$; we begin with $X_0=n$.\\

We break up the travel time to $X_\star$ into three checkpoints: $2nx_\star$, $nx_\star+\sqrt{n}$, and $X_\star$.\\

\noindent\textit{First checkpoint.} If $x>x_\star$ then
$$\mu(x-x_\star) = \mu(x) = x(r(1-x)-1) = rx(x_\star-x) \le -r(x-x_\star)^2 \le -(x-x_\star)^2.$$
The differential equation $y'=-y^2$ has solution flow $\phi(t,y) = 1/(1/y+t)$, so letting $\tau_1=\inf\{t\colon x_t \le 2x_\star\}$ and defining the continued process $\tld x$ by
$$\tld x_t = \phi(t - t\wedge \tau_1,x_{t \wedge \tau_1}),$$
we have $\mu_t(\tld x) \le -\tld x_t^2$ for all $t\ge 0$. Taking expectations and using Jensen's inequality,
$$\frac{d}{dt}\E[\tld x_t] \le - \E[\tld x_t^2] \le -(\E[ \tld x_t])^2,$$
which gives $\E[\tld x_t] \le \phi(t,x_0) = 1/(1/x_0+t)$. Since $x_0\le 1$, Markov's inequality then gives
$$\P(\tau_1>t) \le \P(\tld x_t > 2x_\star) \le \frac{1}{2x_\star}\frac{1}{1+t} = \frac{r}{2\dlt(1+t)}.$$
Letting $t=1/\dlt$ the above is at most $1/2$. Using the Markov property and iterating, $\P(\tau_1 > k/\dlt) \le 2^{-k}$, so 
$\E[\tau_1] \le (1/\dlt) \sum_{k \ge 0}\P(\tau_1 > k/\dlt) \le 2/\dlt$.\\

\noindent\textit{Second checkpoint.} To get from $2x_\star$ to $x_\star+1/\sqrt{n}$, let $\tau_2=\inf\{t \colon x_t \le x_\star+1/\sqrt{n}\}$ and note that if $x>x_\star$ then
$$\mu(x-x_\star) = -rx(x-x_\star) \le -rx_\star(x-x_\star) = -\dlt(x-x_\star).$$
Thus $\xi_t = e^{\dlt (t\wedge \tau_2)}(x_{t \wedge \tau_2}-x_\star)$ is a supermartingale and if $x_0 \le 2x_\star$ then $\xi_0 \le x_\star$ and
$$\P(\tau_2 > t \mid x_0 \le 2x_\star) \le \P(\xi_t \ge e^{\dlt t}/\sqrt{n}) \le e^{-\dlt}\sqrt{n}x_\star \le e^{-\dlt t}\sqrt{n}\dlt = e^{-\dlt t}c.$$
Thus $\E[\tau_2 \mid x_0 \le 2x_\star] = \int_0^{\infty}\P(\tau_2>t\mid x_0\le 2x_\star)dt \le c/\dlt$.\\

\noindent\textit{Third checkpoint.} Finally we estimate $\tau_3=\inf\{t \colon X_t=X_\star\}$, assuming $X_0 \le nx_\star+\sqrt{n}$.\\
With $S_-(j)$ as in \eqref{eq:sj-dn},
\begin{align}\label{eq:tau3}
\E[\tau_3 \mid X_0 \le nx_\star + \sqrt{n} ] \le \sum_{j=X_\star+1}^{X_\star+\sqrt{n}+1}S_-(j).
\end{align}
Since $h_+(i)=h_+(j)=1$, $q_+(j) \le q_-(j)=q_-^\star(j)$ and $q_-(i)=i \ge X_\star$ for $i,j>X_\star$,
$$S_-(j) \le \frac{1}{X_\star}\sum_{i=j}^n \nu(i-1,j).$$
If $X_\star<j \le X_\star+\sqrt{n}+1$ then using Lemma \ref{lem:nu2} with $\sigma \le 1 + 1/\sqrt{n}$,
$$\nu(i-1,j) = \frac{\nu(i-1,X_\star)}{\nu(j,X_\star)} \le (1+o(1))e^{r/2}\nu(i-1,X_\star).$$
Using again Lemma \ref{lem:nu2} and approximating the sum by a Gaussian integral we obtain
$$S_-(j) \le \frac{1}{X_\star}(1+o(1))e^{r/2}\sqrt{\frac{n\pi}{2r}}=(1+o(1))\sqrt{\frac{\pi}{2r}}e^{r/2}\frac{\sqrt{n}}{n\dlt} = O(1/\sqrt{n}\dlt),$$
since $r$ is bounded by assumption. Summing over $\sqrt{n}+1$ terms of the same size and using \eqref{eq:tau3}, we find
$$\E[\tau_3 \mid X_0 \le nx_\star + \sqrt{n} ] = O(\sqrt{n}/\sqrt{n}\dlt) = O(1/\dlt).$$
In all three cases the expected travel time is $O(c/\dlt)$, which satisfies the stated estimates - for the case $\dlt \to \dlt_\infty>0$, note that $c/\dlt = \sqrt{n}$.\\

Next we consider the case $X_0=1$.\\
From \eqref{eq:sj-up}, and since $q_-(j) \le q_+(j) \le q_+^\star(j)$ and $q_+(i) \ge i$ for $i,j<X_\star$,
$$S_+^\star(j) \le \sum_{i=1}^j \nu(i,j-1) \frac{h_+(i)^2}{h_+(j)^2}\,\frac{1}{i}.$$
Let $s_{ij}$ denote the above summands. Then,
$$\E[\tau_\star \mid X_0=1, \ X_{\tau_\star}=X_\star] = \sum_{j=1}^{X_\star-1}S_+^\star(j) \le \sum_{j=1}^{X_\star-1}\sum_{i=1}^j s_{ij}.$$
If $i\le j <X_\star$ then $\nu(i,j-1) \le r$ and $h_+(i)/h_+(j) \le 1$, which we use below. In order to obtain good enough estimates, we need to be a bit more precise. We treat the cases $\dlt \to 0$ and $\dlt \to \dlt_\infty>0$ separately.\\

\noindent\textbf{Case 1: $\dlt \to 0$.} Let $c=\sqrt{n}\dlt$, so $c\to\infty$ and $c=o(\sqrt{n})$. We treat the sum in three parts:
\begin{enumerate}[noitemsep,label={\roman*)}]
\item $1 \le i \le j \le 1/\dlt$, 
\item $1 \le i \le 1/\dlt < j \le X_\star-1$ and
\item $1/\dlt < i \le j < X_\star$.
\end{enumerate}
\noindent\textit{Part i).} From \eqref{eq:hp-ub}, $h_+(i) \le \dlt i$, and since the denominator $\sim 1/(1-1/r)$,
$$h_+(j) = \frac{\sum_{k=0}^{X_0-1}\nu(k)}{\sum_{k=0}^{X_\star-1}\nu(k)} \ge \frac{(1-r^{-j})/(1-1/r)}{(1+o(1)/(1-1/r)} = (1-o(1))(1-r^{-j}).$$
Since $\dlt\to 0$, $(1+\dlt)^{-j} = ((1+\dlt)^{1/\dlt})^{-\dlt j} \to e^{-\dlt j}$ uniformly over $\dlt j \le 1$. Since $e^{-x} \le 1-(1-1/e)x$ for $x \in [0,1]$, if $\dlt j \le 1$ then
\begin{align}\label{eq:hp-lb}
h_+(j) \ge (1-o(1))(1-e^{-\dlt j}) \ge (1-o(1))(1-1/e)\dlt j
\end{align}
which is at least $\dlt j/2$ for large $n$, since $1-1/e > 1/2$. Thus, $s_{ij} \le r(i^2/(j/2)^2)(1/i) = 4ri/j^2$, so
$$\sum_{j=1}^{\lf 1/\dlt \rf}\sum_{i=1}^j s_{ij} \le \sum_{j=1}^{\lf 1/\dlt \rf}\frac{4r}{j^2}\sum_{i=1}^j i \le \sum_{j=1}^{\lf 1/\dlt \rf} 2r \le 2r/\dlt.$$
\noindent\textit{Part ii).} Since $q_-(k)/q_+(k)\ge 1/r$, $\nu(0,i) \ge r^{-i}=(1+\dlt)^{-i} \ge e^{-\dlt i}$ for each $i$. Since $q_-(k)\le q_+(k)$ for $k<X_\star$ and $i\le \lf 1/\dlt \rf$, $\nu(i,j-1) \le \nu(\lf 1/\dlt \rf,j-1)$, so
$$\nu(i,j-1) = \frac{\nu(0,j-1)}{\nu(0,i)} \le e^{\dlt i}\nu(0,j-1).$$
Thus if $i \le 1/\dlt < j < X_\star$ then
$$\nu(i,j-1) = \frac{\nu(0,j-1)}{\nu(0,i)} \le e^1\nu(0,j-1).$$
Since $j>\lf 1/\dlt \rf$ and $j\mapsto h_+(j)$ is non-decreasing, using \eqref{eq:hp-lb}, $h_+(j) \ge (1-o(1))(1-1/e)\dlt \lf 1/\dlt \rf$ is at least $1/2$ for large $n$, since $\dlt \to 0$ implies $\dlt\lf 1/\dlt \rf \to 1$. Using again $h_+(i) \le \dlt i$ and combining,
$$\sum_{j=\lf 1/\dlt \rf+1}^{X_\star-1}\sum_{i=1}^{\lf 1/\dlt \rf} s_{ij} \le \sum_{j=\lf 1/\dlt \rf+1}^{X_\star-1} e^1 \nu(0,j-1) \sum_{i=1}^{\lf 1/\dlt \rf} \frac{(\dlt i)^2}{1/4}\frac{1}{i}.$$
We easily estimate
$$\sum_{i=1}^{\lf 1/\dlt \rf} \frac{(\dlt i)^2}{1/4}\frac{1}{i} \le 4\dlt^2\sum_{i=1}^{\lf 1/\dlt \rf}i \le 4\dlt^2\frac{(1/\dlt)^2}{2} = 2.$$
Using \eqref{eq:nu-sum0},
$$\sum_{j=\lf 1/\dlt \rf +1}^{X_\star -1}e^1\nu(0,j-1) \le (1+o(1))/\dlt.$$
Combining the two, the sum is at most $(2+o(1))/\dlt$.\\

\noindent\textit{Part iii).} This part is the easiest; we simply use $\nu(i,j-1) \le r$, $h_+(i)/h_+(j) \le 1$ and $1/q_+(i) \le 1/i$ and treating the sum as a right-endpoint Riemann sum,
$$\sum_{j=\lf 1/\dlt \rf +1}^{X_\star-1}\sum_{i=\lf 1/\dlt \rf +1}^j \frac{1}{i} \le \sum_{j=\lf 1/\dlt \rf +1}^{X_\star-1}r(\log(j)-\log(1/\dlt)).$$
We can combine the logs as $\log(\dlt j)$, which is increasing in $j$. Treating the sum as a left-endpoint Riemann sum of the function $\log(x)$ with interval widths $\dlt$ and noting $\dlt X_\star \le n\dlt^2 = c^2$, the sum is at most
$$\frac{1}{\dlt}(\dlt X_\star(\log(\dlt X_\star) - 1) - \dlt(1/\dlt)(\log(\dlt(1/\dlt))-1) \le \frac{1}{\dlt}(c^2(\log(c^2)-1) + 1).$$
Combining all three parts, we find
$$\E[\tau_\star \mid X_0=1, \ X_{\tau_\star}=X_\star] \le \frac{1}{\dlt}(c^2\log(c^2) - c^2 + O(1)) \le \frac{1}{\dlt} 2c^2\log(c)$$
for large $n$, since $c\to\infty$.\\

\noindent\textbf{Case 2: $\dlt \to \dlt_\infty >0$.} Since $1/\dlt=O(1)$ in this case, the whole sum can be treated as in part iii) above. Since $s_{ij} \le r/i$, 
$$\sum_{j=1}^{X_\star-1}\sum_{i=1}^j s_{ij} \le \sum_{j=1}^{X_\star-1}(1+\log j).$$
Treating the sum as a left-endpoint Riemann sum, it is at most
$$X_\star \log(X_\star) - 1\log(1) \le n\dlt \log(n \dlt) = O(n\log(n)).$$
\end{proof}

\subsubsection{Fall time}

Here we show that the time to hit zero after the last visit to $X_\star$ is small compared to the sojourn time. The following is an equivalent formulation of Lemma \ref{lem:fall}.

\begin{lemma}\label{lem:fall-time}
$\E[\tau \mid X_0=X_\star, X_{\tau_\star}=0] = o(E_\star^o)$.
\end{lemma}

\begin{proof}
Let $L_\star^0$ denote the above expectation. With $S_-^0(j)$ as in \eqref{eq:sj0-dn},
$$L_\star^0 = \sum_{j=1}^{X_\star}S_-^0(j).$$
Since we condition on $X_{\tau_\star}=0$, the initial jump off $X_\star$ is to $X_\star-1$ with rate $q_+(X_\star)+q_-(X_\star)$ that we denote $q_-^0(X_\star)$, after which we use the rates $q_\pm^0$ given by \eqref{eq:cond-trans}. Thus $S_{X_\star} = 1/q_-^0(X_\star)$. For $j\le X_\star-1$, $q_+^0(j) \le q_+(j)$ and $q_-^0(j) \ge q_-(j)$, so $q_+^0(j)/q_-^0(j) \le q_+(j)/q_-(j) \le r$. Moreover, $q_-^0(i) \ge q_-(i)=i$, so
$$S_-^0(j) \le r\sum_{i=j}^{X_\star-1}\nu^0(i-1,j)\, \frac{1}{i}.$$
Let $s_{ij}$ denote the summands. Summing over $j$ and exchanging the order of summation,
$$L_\star^0 = \frac{1}{q_-^0(X_\star)} + r\sum_{i=1}^{X_\star-1}\sum_{j=1}^i s_{ij}.$$
The first term is at most $1/X_\star$ which is clearly $o(E_\star^o)$. To estimate the sum we need more information about $\nu^0$, so we first estimate the ratios
\begin{align}\label{eq:cond-down-rates}
\frac{q_+^0(j)}{q_-^0(j)} = \frac{q_+(j)}{q_-(j)} \, \frac{h_-(j+1)/h_-(j)}{h_-(j-1)/h_-(j)}
\end{align}
of the conditioned rates given by \eqref{eq:cond-trans}. To do so we use the formulas \eqref{eq:sdie-hp} and \eqref{eq:sdie-hm}. Since $\nu(j-1,k) \le \nu(j,k)$ for $j < X_\star$, from \eqref{eq:sdie-hm},
\begin{align}\label{eq:sdie-hm1}
\frac{h_-(j-1)}{h_-(j)} \ge 1 + \frac{1}{\sum_{k=j}^{X_\star-1}\nu(j,k)}
\end{align}
which simplifies some calculations. Define $\sigma_j=(X_\star-j)/\sqrt{n}$ and similarly for $\sigma_k$, and let $\Sigma=n^{1/8}\wedge \sqrt{c}$.\\

\noindent\textit{Estimation for $\sigma_j \le \Sigma.$} By Lemma \ref{lem:nu2}, uniformly over $0 \le \sigma_k \le \sigma_j \le n^{1/8}$,
$$\nu(j,k) = \nu(j,X_\star)/\nu(k,X_\star) \sim e^{(-\sigma_j^2 + \sigma_k^2)r/2},$$
and so
$$\sum_{k=j}^{X_\star-1}\nu(j,k) \sim \sqrt{n}\int_0^{\sigma_j} e^{-(\sigma_j-\sigma_k)(\sigma_j+\sigma_k)r/2}d\sigma_k.$$
Changing variables to $u=\sigma_j-\sigma_k$, $\sigma_j+\sigma_k = 2\sigma_j-u$ and the integral becomes
$$\int_0^{\sigma_j} e^{-u(2\sigma_j-u)r/2}du \le \int_0^{\sigma_j} e^{-u\sigma_j r/2}du \le \frac{2}{\sigma_j r}.$$
Letting $b_j = (1-\ep_n)\sigma_j r/\sqrt{n}$ with $\ep_n\to 0$ sufficiently slowly and using \eqref{eq:sdie-hp} and \eqref{eq:sdie-hm1} it follows that
\begin{align*}
\frac{h_-(j-1)}{h_-(j)} &\ge 1 + b_j/2 \ \text{and} \\
\frac{h_-(j+1)}{h_-(j)} &\le 1 - b_j/2.
\end{align*}
Since $\sigma_j \le n^{1/8}$, $b_j=o(1)$ so $h_-(j+1)/h_-(j-1) \le 1-b_j+o(b_j)$. On the other hand
\begin{align}\label{eq:q-ud-ratio}
\frac{q_+(j)}{q_-(j)} &= r(1-j/n) = (1 + r(x_\star-j/n)) \nonumber \\
&= 1 + r\sigma_j/\sqrt{n} + O(1/n) = 1 + (1+o(1))b_j+O(1/n).
\end{align}
Using \eqref{eq:cond-down-rates} and the above estimates,
\begin{align}\label{eq:nu0-small}
\frac{q_+^0(j)}{q_-^0(j)} &\le (1+(1+o(1))b_j+ O(1/n))(1-b_j+O(b_j^2)) = 1-b_j^2 + o(b_j^2) + O(1/n).
\end{align}
\noindent\textit{Estimation for $\sigma_j \ge \Sigma$}. Recall the upper bound from Lemma \ref{lem:nu}:
$$\nu(j,k) \le \exp(-n(V((k+1)/n)-V((j+1)/n)).$$
Using the fact that $V$ is non-decreasing and $V'$ is non-increasing on $[0,x_\star]$,
$$n(V((k+1)/n)-V((j+1)/n)) \ge ((k - j ) \wedge \sqrt{n})V'((j+1+\sqrt{n})/n).$$
With this bound,
\begin{align}\label{eq:nu-sum}
\sum_{k=j}^{X_\star-1}\nu(j,k) \le \frac{1}{1- e^{-V'((j+1+\sqrt{n})/n)}} + (X_\star-j)e^{-\sqrt{n}V'(j+1+\sqrt{n}/n)},
\end{align}
the first at most $\sqrt{n}$ terms forming a partial geometric series, and the last at most $X_\star-j$ terms each contributing at most a constant. Since $V'(x)=\log(r(1-x))$ and $r(1-x) = 1 + r(x_\star-x)$,
\begin{align*}
e^{V'((j+1+\sqrt{n})/n)} &= 1 + r(x_\star-(j+1+\sqrt{n})/n) \\
&= 1 + r\big((X_\star-j)/n - 1/\sqrt{n} + O(1/n)) \\
&\ge 1 + r(\sigma_j-2)/\sqrt{n}
\end{align*}
for large $n$. This easily gives the bound $\sqrt{n}/(r(\sigma_j-2))$ on the first term on the RHS of \eqref{eq:nu-sum}. To bound the second term note that $\sigma_j \le \sqrt{n}$ and $r\ge 1$, and that $1+x\ge e^{x/2}$ for $x\in[0,1]$, so $1+r(\sigma_j-2)/\sqrt{n} \ge e^{(\sigma_j-2)/2\sqrt{n}}$. Using this on the second term on the RHS of \eqref{eq:nu-sum} and combining the two estimates, for large $n$
$$\sum_{k=j}^{X_\star-1}\nu(j,k) \le \frac{\sqrt{n}}{r(\sigma_j-2)} + \sqrt{n}\sigma_je^{-(\sigma_j-2)/2}.$$
Since $xe^{-x/2} \to 0$ faster than $1/x$ as $x\to\infty$, using $b_j=(1-\ep_n)r\sigma_j/\sqrt{n}$ with $\ep_n\to 0$ slowly enough, since $\Sigma\to\infty$ it follows that uniformly over $\sigma_j \ge \Sigma$,
$$\sum_{k=j}^{X_\star-1}\nu(j,k) \le 1/b_j,$$
and so
$$\frac{h_-(j+1)}{h_-(j-1)} \le \frac{1-b_j}{1+b_j}.$$
Since $\sigma_j \ge \Sigma\to\infty$, $b_j=\omega(1/n)$, and using \eqref{eq:q-ud-ratio}, $q_+(j)/q_-(j) = 1 + (1+o(1))b_j$. Using \eqref{eq:cond-down-rates}, uniformly over $j \le X_\star-\Sigma\sqrt{n}$,
\begin{align}\label{eq:nu0-big}
\frac{q_+^0(j)}{q_-^0(j)} \le (1+(1+o(1))b_j)\frac{1-b_j}{1+b_j} = 1-b_j + o(b_j).
\end{align}

\noindent\textbf{Case 1: $\dlt \to 0$.} Since $\nu^0(i,i-1) \to 1$ uniformly over $i$, we can work with $\nu^0(i,j)$ instead of $\nu^0(i-1,j)$. Again, we break the sum into parts; the decomposition is similar to the one in the second half of the proof of Lemma \ref{lem:equil-time}, except that the third part has been further subdivided into three parts, for a total of five:
\begin{enumerate}[noitemsep,label={\roman*)}]
\item $1 \le j \le i \le 1/\dlt$, 
\item $1 \le j \le 1/\dlt < i \le X_\star-\Sigma\sqrt{n}$,
\item $1/\dlt < j \le i \le X_\star-\Sigma\sqrt{n}$,
\item $1/\dlt < j \le X_\star-\Sigma\sqrt{n} < i \le X_\star-1$, and
\item $X_\star-\Sigma\sqrt{n} < j \le i \le X_\star-1$.
\end{enumerate}
Note that $\sigma_i \ge \Sigma$ in parts i-iii and $\sigma_j \ge \Sigma$ in parts i-iv.\\

\noindent\textit{Part i).} Note that if $j = o(X_\star)$, which is the case if $j\le 1/\dlt$, then $\sigma_j \sim \sqrt{n}x_\star$ and $b_j \sim \dlt$, so
$$\nu^0(i,j) \le (1-\dlt+o(\dlt))^{i-j} \le e^{-(1+o(1))\dlt(i-j)},$$
and treating as a partial geometric sum,
$$\sum_{j=1}^i\nu^0(i,j) \le \frac{1-e^{-(1+o(1))\dlt i}}{1-e^{-(1+o(1))\dlt}} \le \frac{(1+o(1))\dlt i}{\dlt}=(1+o(1))i.$$
Thus
$$\sum_{i=1}^{\lf 1/\dlt \rf}\sum_{j=1}^i s_{ij} \le \sum_{i=1}^{\lf 1/\dlt \rf}\frac{1}{i}(1+o(1))i \le (1+o(1))/\dlt.$$
\noindent\textit{Part ii).} If $i>1/\dlt$ then from the above,
$$\sum_{j=1}^{\lf 1/\dlt \rf}\nu^0(i,j) \le \frac{1}{1-e^{-(1+o(1))\dlt}} \le (1+o(1))/\dlt.$$
Noting that $X_\star \le n\dlt = \sqrt{n}c$ and $1/\dlt = \sqrt{n}/c$,
$$\sum_{i=\lf 1/\dlt \rf+1}^{X_\star-\Sigma\sqrt{n}}\sum_{j=1}^{\lf 1/\dlt \rf} s_{ij} \le \frac{1+o(1)}{\dlt}(\log(\sqrt{n}c)-\log(\sqrt{n}/c)) = \frac{2+o(1)}{\dlt}\log(c).$$
\noindent\textit{Part iii).} Since $i\ge j$, $b_j = (1-\eps_n)r(X_\star-j)/n \le (1-\eps_n)r(X_\star-i)/n = b_i$, $\nu^0(i,j) \le e^{-(1-o(1))b_i(i-j)}$ and
$$\sum_{j \le i}\nu^0(i,j) \le 1/(1-e^{-(1-o(1))b_i}) \sim 1/b_i = n/(X_\star -i)$$
uniformly over $i$, since $b_i \le \dlt$ and $\dlt \to 0$. This gives
$$\sum_{i =\lf 1/\dlt \rf+1}^{\lf X_\star - \Sigma\sqrt{n} \rf}\sum_{j=\lf 1/\dlt \rf+1}^i s_{ij}
\le \sum_{i =\lf 1/\dlt \rf+1}^{\lf X_\star - \Sigma\sqrt{n} \rf}\frac{n}{i(X_\star-i)} 
\le \frac{n}{X_\star}\sum_{i =\lf 1/\dlt \rf+1}^{\lf X_\star - \Sigma\sqrt{n} \rf}\big(\frac{1}{i} + \frac{1}{X_\star-i}\big).$$
Treating the sums as Riemann sums and noting $n/X_\star \le n/(n\dlt-1) \sim 1/\dlt$, $\dlt X_\star\le n\dlt^2=c^2$ and $X_\star/(\Sigma\sqrt{n})\le \sqrt{n}\dlt/\Sigma = o(\sqrt{c})$, this is at most
$$\frac{1+o(1)}{\dlt}\left(\log(X_\star)-\log(1/\dlt) + \log(X_\star) - \log(\Sigma\sqrt{n})\right) \le \frac{1+o(1)}{\dlt}(c^2 + o(\sqrt{c})).$$
\noindent\textit{Part iv).} Writing as a product and using \eqref{eq:nu0-small} on the first term, then proceeding as in part iii) on the sum, for $i\ge X_\star-\Sigma\sqrt{n} \ge j$,
\begin{align*}
\sum_{j \le X_\star-\Sigma\sqrt{n}}\nu^0(i,j) 
&= \nu^0(i,X_\star-\Sigma\sqrt{n})\sum_{j \le X_\star-\Sigma\sqrt{n}}\nu^0(X_\star-\Sigma\sqrt{n},j) \\
&\le e^{n^{5/8}O(1/n)}/(1 - e^{-(1-o(1))\Sigma/\sqrt{n})} = (1+o(1))\sqrt{n}/\Sigma.
\end{align*}
Since $\Sigma=o(c)=o(X_\star/\sqrt{n})$, $1/i \sim 1/X_\star$ and
\begin{align*}
\sum_{i=X_\star-\Sigma\sqrt{n}+1}^{X_\star-1}\sum_{j=\lf 1/\dlt \rf+1}^{X_\star-\Sigma\sqrt{n}}s_{ij} 
&\le (1+o(1))\frac{\sqrt{n}}{\Sigma}\sum_{i=X_\star-\Sigma\sqrt{n}+1}^{X_\star-1}\frac{1}{i} \\
&\le (1+o(1))\frac{\sqrt{n}}{\Sigma}\frac{\Sigma\sqrt{n}}{X_\star} \sim \frac{n}{X_\star} \le 1/\dlt.
\end{align*}

\noindent\textit{Part v).} Using \eqref{eq:nu0-small} as in part iv), $\nu^0(i,j) \le e^{\Sigma\sqrt{n}O(1/n)}=1+o(1)$. Again, $1/i \sim 1/X_\star$. Since there are at most $\Sigma^2n$ terms in the sum and $\Sigma \le \sqrt{c}$, it is bounded by
$$(1+o(1))\frac{\Sigma^2n}{X_\star} = (1+o(1))\frac{c}{\dlt}.$$
In all five parts, the sum is $O(c^2/\dlt)$; referring to \eqref{eq:Estar-est}, this is $o(E_\star^o)$.\\

\noindent\textbf{Case 2: $\dlt \to \dlt_\infty>0$.} Since by \eqref{eq:nu0-big}, $q_+^-(j)/q_-^0(j) \le 1$ for $j\le X_\star-\Sigma\sqrt{n}$ and by \eqref{eq:nu0-small}, $q_+^0(j)/q_-^0(j) \le 1 + O(1/n)\le e^{O(1/n)}$ for $X_\star-\Sigma\sqrt{n} \le j < X_\star$, and since $\Sigma \le n^{1/8}$, $\nu^0(i,j) \le e^{n^{5/8}O(1/n)} = 1+o(1)$ for $j \le i < X_\star$. Thus,
$$\sum_{i=1}^{X_\star-1}\sum_{j=1}^i s_{ij} \le (1+o(1)\sum_{i=1}^{X_\star-1}\frac{1}{i} \, i \sim X_\star,$$
and $X_\star \le n\dlt = o(E_\star^o)$, again by \eqref{eq:Estar-est}.
\end{proof}

\subsubsection{Exponential limit}\label{sec:exp-limit}

Here we prove Lemma \ref{lem:exp}. By the strong Markov property, this is equivalent to showing that $\tau_\star^o/E_\star^o$ converges in distribution to exponential with mean 1, assuming $X_0=X_\star$. Let $\Phi$ denote the natural coupling, so that for each $j\in \{0,\dots,N\}$, $((\Phi(j,t))_{t\ge 0}$ is a copy of the logistic process with initial value $j$, and by Lemma \ref{lem:nat-coup}, $\Phi(i,t)\le \Phi(j,t)$ for all $t$ if $i\le j$, and let
$$\tau_\star(j)=\inf\{t>0 \colon \Phi(j,t)\in \{0,X_\star\}\} \quad \text{and} \quad \tau_\star^o(j) = \sup\{t>0 \colon \Phi(j,t)=X_\star\}.$$
We give a sufficient condition for $(\tau_\star^o-\tau_\star)/E_\star^o$ to have an exponential limit.

\begin{lemma}\label{lem:suff-exp}
Let $\ol E_\star^o = \E[\tau_\star^o(n)]$ and assume that
\begin{enumerate}[noitemsep]
\item $\P(\tau_\star^o(j) = \tau_\star^o(n) \mid \Phi(j,\tau_\star(j))=X_\star) = 1-o(1)$ uniformly over $j \in \{1,\dots,n\}$ and
\item uniformly over $j \in \{1,\dots,n\}$ and $t>0$,
$$\P(\tau_\star^o(n)> t \ol E_\star^o \mid \Phi(j,\tau_\star(j))=X_\star) \ge \P(\tau_\star^o(n)> t \ol E_\star^o)-o(1).$$
\end{enumerate}
Then for each $t>0$, $\P(\tau_\star^o(X_\star) > tE_\star^o) \to e^{-t}$.
\end{lemma}

\begin{proof}
By definition, $\Phi(X_\star,\tau_\star^o(X_\star))=X_\star$, so using assumption 1, it is enough to show that $\P(\tau_\star^o(n) > tE_\star^o) \to e^{-t}$ for $t>0$. By definition of $\tau_\star^o(n)$, $\ol E_\star^o = \E[\tau_\star^o \mid X_0=n]$, so using the Markov property and then Lemma \ref{lem:appr} we find that $\ol E_\star^o = E_\star^o + \E[\tau_\star \mid X_0=n] \sim E_\star^o$. Thus it is enough to show that $p_n(t) := \P(\tau_\star^o(n) > t\ol E_\star^o) \to e^{-t}$ for $t>0$.\\

Using the natural coupling of Section \ref{sec:coup}, if $j\le n$ then for any $t>0$, $\Phi(j,t) \le \Phi(n,t)$ which implies $\tau_\star^o(j) \le \tau_\star^o(n)$. Conditioning on the value of $\Phi(n,t)$ and using the Markov property, it follows that
$$p_n(t+s) = p_n(t)\sum_j \P(\tau_\star^o(j) > s\ol E_\star^o)\P(\Phi(n,t \ol E_\star^o)=j) \le p_n(t)p_n(s),$$
i.e., $t\mapsto p_n(t)$ is submultiplicative for each $n$. Given $t,s>0$, conditioning on $\Phi(n,t \ol E_\star^o)$, using the Markov property, then using assumption 1 then assumption 2, then the law of total probability,
\begin{align*}
\P(\tau_\star^o(n)>(t+s)\ol E_\star^o \mid \tau_\star^o(n) > t \ol E_\star^o) &= \sum_j \P(\tau_\star^o(j)>s \ol E_\star^o) \mid \Phi(j,\tau_\star(j))=X_\star)\P(\Phi(n,t \ol E_\star^o)=j) \\
&\ge \sum_j \P(\tau_\star^o(n)>s \ol E_\star^o) \mid \Phi(j,\tau_\star(j))=X_\star)\P(\Phi(n,t \ol E_\star^o)=j)-o(1) \\
&\ge \sum_j \P(\tau_\star^o(n)>s \ol E_\star^o))\P(\Phi(n,t \ol E_\star^o)=j)-o(1) \\
&= p_n(s)-o(1)
\end{align*}
uniformly over $t$ and $s$. Since the above LHS is just $p_n(t+s)/p_n(t)$, we obtain $p_n(t+s) \ge p_n(t)p_n(s)-o(1)$. Combining with $p_n(t+s)\le p_n(t)p_n(s)$ it follows easily that for rational $t$, $p_n(t) = p_n(1)^t + o(1)$, and since $t\mapsto p_n(t)$ is non-increasing, the same holds for real $t>0$ by rational approximation. Thus it remains only to show that $p_n(1)\to 1/e$ as $n\to\infty$.\\

Let $T(n) = \tau_\star^o(n)/\ol E_\star^o$, so that $\E[T(n)]=1$ for each $n$. By Markov's inequality, $p_n(2) \le 1/2$, and $p_n(2j) \le p_n(2)^j = 2^{-j}$ for integer $j\ge 1$, so
$$\E[\, T(n) \1(T(n) > 2k) \, ] \le \sum_{j \ge k}2\P(T_\star^o > 2j) \le 2^{-(k-2)},$$
which $\to 0$ as $k\to\infty$ uniformly in $n$. Since $\P(T(n)>t) = \P(T(n)>1)^t + o(1)$ for each $t$, an easy approximation argument using the monotonicity of $t\mapsto \P(T(n)>t)$ then shows that
$$\E[ \, T(n) \, ] = \int_0^{\infty} \P(T(n)>1)^t dt + o(1).$$
Since $\E[ \, T(n) \, ] = 1$ for all $n$, it follows that $\P(T(n)>1) \to 1/e$ as $n\to\infty$, as desired.
\end{proof}

It remains to show assumptions 1 and 2 of Lemma \ref{lem:suff-exp} are satisfied. Let $\tau_c(j) = \inf\{t \colon \Phi(j,t)=\Phi(n,t)\}$ denote the coupling time of the two trajectories, which is a.s.~finite since both eventually hit 0. We begin by extracting a further sufficient condition, which we then prove.

\begin{lemma}\label{lem:suff-exp2}
Let $\tau_c^\star(j) = \inf\{t \ge \tau_c(j) \colon \Phi(j,t) = X_\star\}$. Suppose that
\begin{align}\label{eq:couple-hit}
\min_{j \in \{1,\dots,n\}}\P(\tau_c^\star(j) <\infty \mid \Phi(j,\tau_\star(j))=X_\star) \to 1.
\end{align}
Then assumptions 1 and 2 of Lemma \ref{lem:suff-exp} are satisfied.
\end{lemma}

\begin{proof}
\textit{Assumption 1.} Since $\Phi(j,t)=\Phi(n,t)$ for all $t\ge \tau_c(j)$ (see Lemma \ref{lem:nat-coup}), the event $\tau_\star^o(j) = \tau_\star^o(n)$ is equivalent to the event that $\Phi(j,t)=\Phi(n,t)=X_\star$ for some $t>0$, which in turn is equivalent to the event $\tau_c^\star(j)<\infty$, and assumption 1 follows directly from \eqref{eq:couple-hit}.\\

\noindent \textit{Assumption 2.} First note that, since $\P(\Phi(X_\star,\tau_\star(X_\star)) = X_\star)=1$, it follows from the above that
\begin{align}\label{eq:n-X0-lasthit}
\P(\tau_\star^o(X_\star)=\tau_\star^o(n)) = 1-o(1).
\end{align}
Next, using the strong Markov property, $\tau_\star^o(n)$ conditioned on $\tau_c^\star(j)<\infty$ is equal in distribution to $\tau_c^\star(j)$ conditioned on $\tau_c^\star(j)<\infty$, plus an independent copy of $\tau_\star^o(X_\star)$. In particular, $\tau_\star^o(n)$, conditioned on $\tau_c^\star(j)<\infty$, dominates $\tau_\star^o(X_\star)$ (with no conditioning). Since $\tau_c^\star(j)<\infty$ implies $\Phi(j,\tau_\star(j))=X_\star$, using \eqref{eq:couple-hit}, then the above observation, then \eqref{eq:n-X0-lasthit}, it follows that uniformly over $j$,
\begin{align*}
\P(\tau_\star^o(n) > t \ol E_\star^o \mid \Phi(j,\tau_\star(j)) = X_\star)
& = \P(\tau_\star^o(n) > t \ol E_\star^o \mid \tau_c^\star(j)<\infty )\P(\tau_c^\star(j)<\infty \mid \Phi(j,\tau_\star(j))=X_\star) \\
&= \P(\tau_\star^o(n) > t \ol E_\star^o \mid \tau_c^\star(j)<\infty )(1-o(1)) \\
& \ge \P(\tau_\star^o(X_\star) > t \ol E_\star^o)(1-o(1)) \\
&= (\P(\tau_\star^o(n) > t \ol E_\star^o) - o(1))(1-o(1)) \\
&= \P(\tau_\star^o(n) > t \ol E_\star^o) - o(1).
\end{align*}
\end{proof}

Finally we prove the hypothesis of Lemma \ref{lem:suff-exp2}. To do so we show that within a short time after the paths started from $j$ and from $n$ reach $X_\star$, they meet (if they have not met already), and then with probability $1-o(1)$ their common trajectory hits $X_\star$ at least once more before going to $0$.

\begin{lemma}
As $n\to\infty$,
\begin{align}\label{eq:couple-hit2}
\min_{j \in \{1,\dots,n\}}\P(\tau_c^\star(j) <\infty \mid \Phi(j,\tau_\star(j))=X_\star) \to 1.
\end{align}
\end{lemma}

\begin{proof}
Let $X^j,X^n$ denote the processes $(\Phi(j,t))_{t\ge 0},(\Phi(n,t))_{t\ge 0}$. As we will see, when $X^j,X^n>(1+\eps)X_\star/2$, the drift tends to push them together. Let
$$\tau_b(j) = \inf\{t>\tau_\star(j) \vee \tau_\star(n)\colon \max(|X^j_t-nx_\star|,|X^n_t -nx_\star|) \ge nx_\star/4\}.$$
If $\tau_c(j) < \tau_\star(j) \vee \tau_\star(n)$ then $X_t^j=X_t^n=X_\star$ with $t=\tau_\star(j) \vee \tau_\star(n)$ which implies $\tau_c^\star(j)<\infty$. On the other hand, if $\tau_\star(j) \vee \tau_\star(n) < \tau_c(j) < \tau_b(j)$ then $X_{\tau_c(j)}^j=X_{\tau_c(j)}^n \ge 3X_\star/4$. Using the strong Markov property and the fact that $X^j$ and $X^n$ remain together once they meet, on the latter event it follows from Lemma \ref{lem:meta-thr} that $\tau_c^\star(j)<\infty$ with probability $1-o(1)$ uniformly over $j$. Thus it is enough to show that
\begin{align}\label{eq:coup-time}
\max_j \P(\tau_b(j) \wedge \tau_c(j) < \tau_b(j) \mid \Phi(j,\tau_\star(j))=X_\star) \to 1.
\end{align}
We begin with a lower bound on $\tau_b(j)$, that ensures both $X^j,X^n$ remain fairly close to $X_\star$ for a while after they hit it. Then, we estimate the drift and diffusivity of $X^j-X^n$ assuming both are at least $3X_\star/4$, and with the help of the lower bound, deduce \eqref{eq:coup-time}.\\

\noindent\textit{Lower bound on $\tau_b(j)$.} Let $W=X-nx_\star$. Then
\begin{align*}
\mu(W) &= \mu(X) = X(r(1-X/n)-1) = rX(x_\star-X/n) = -rXW/n \ \text{and} \\
\sigma^2(W) &= \sigma^2(X) = X(r(1-X/n)+1) \le (1+r)X.
\end{align*}
Since $W$ jumps by $\pm 1$, if $|W| \ge 1$ then $\mu(|W|) = \sgn(W)\mu(W)$ and $\sigma^2(|W|) = \sigma^2(W)$. Suppose $nx_\star/8 \le |W| \le nx_\star/4$, noting that $1\le nx_\star/8$ for large $n$. Since $3nx_\star/4 \le X \le 5nx_\star/4$,
\begin{align*}
\sigma^2(|W|) &= O(nx_\star) = O(n\dlt), \\
\mu(|W|) &\le -(3rx_\star/4)(nx_\star/8) =  -3n\dlt^2/32r\ \text{and} \\
|\mu(|W|)| &\le (5rx_\star/4)nx_\star/4 = O(n\dlt^2).
\end{align*}
We shall use Lemma \ref{lem:driftbar} (apologies for overloading notation). In the notation of Lemma \ref{lem:driftbar}, let $X=|W|-nx_\star/8$, $x=nx_\star/8 = n\dlt/8r$, $\mu_\star = 3n\dlt^2/32r = 3c^2/32r$, $\sigma^2_\star = Cn\dlt$ and $C_{\mu_\star}=Cn\dlt^2$ for some $C>0$ and $C_\Delta=1/2$. Since $\Delta_\infty(X)=1$, $\Delta_\infty(X)\mu_\star/\sigma^2_\star = 3\dlt/32Cr$ is at most $1/2$ is $C>0$ is chosen large enough. Then, $\Gamma = \exp(\Omega(c^2))$ and $x/16C_{\mu_\star}=\Omega(1/\dlt)$, so
$$\P(\sup_{t \le (1/\dlt)\exp(\Omega(c^2))}|W_t| > nx_\star/4 \ \mid \ |W_0| \le nx_\star/8) = o(1).$$
Applying this bound to $X^j$, $X^n$ from time $\tau_\star(j)$, respectively $\tau_\star(n)$, we find that
\begin{align}\label{eq:apriori}
\P(\tau_b(j) \le (1/\dlt)\exp(\Omega(c^2)) \mid \Phi(j,\tau_\star(j))=X_\star)=o(1)
\end{align}
uniformly over $j\in \{1,\dots,n\}$.\\

\noindent\textit{Upper bound on $\tau_b(j) \wedge \tau_c(j)$.}\\
Let $F(x) = x(r(1-x)-1) = rx(x_\star-x)$ and $G(x)=x(r(1-x)+1) \ge x$, so that
$$\mu(X)=nF(X/n) \quad \text{and} \quad \sigma^2(X) = nG(X/n) \ge X.$$
We have $F'(x) = r(x_\star-2x)$, so if $3x_\star/4 \le x \le 5x_\star/4$ then
$$F'(x) \in [F'(5x_\star/4),F'(3x_\star/4)]=[-3rx_\star/2,-rx_\star/2] = [-3\dlt/2,-\dlt/2].$$
 If $X^j,X^n \ge 3nx_\star/4$ and $X^j \ne X^n$, then letting $U=X^j-X^n$, by the mean value theorem,
$$\mu(U)/(U) = n(F(X^j/n)-F(X^n/n))/(X^j-X^n) \in [-3\dlt/2,-\dlt/2]$$
and since $X^j$ and $X^n$ evolve independently until they collide,
$$\sigma^2(U) = n(G(X^j/n) + G(X^n/n)) \ge 3nx_\star/2 \ge n\dlt/r.$$
Since $U$ jumps by $\pm 1$ and takes values in $\Z$, if $U\ne 0$ then $\mu(|U|)=\sgn(U)\mu(U)$ and $\sigma^2(|U|) = \sigma^2(U)$, so letting $V=|U|$, the above implies that conditional on $\Phi(j,\tau_\star(j))=X_\star$,
$$\mu_t(V)\in [-(3\dlt/2)V_t,-(\dlt /2)V_t] \quad \text{and} \quad \sigma^2_t(V) \ge n\dlt/r$$
for all $\tau_\star(j) \vee \tau_\star(n)\le t<\tau_c(j)\wedge \tau_b(j)$. If this interval is empty, then $\tau_c(j)\wedge \tau_b(j)<\tau_\star(j)\vee \tau_\star(n)\le \tau_b(j)$ so there is nothing to show. Otherwise, since on this time interval, $\sgn(U_t)$ is fixed, then given $\sgn(U_{\tau_\star(j)})$, on the same time interval $V$ is a Markov chain with state space a subset of $\Z$. Let $u =  \lf \sqrt{n} \rf$, $t_0 = \tau_\star(j)$ and $t_1 = \tau_b(j) \wedge \inf\{t>t_0 \colon V_t = u\}$ and define recursively
$$t_i = \tau_b(j) \wedge \inf\{t>t_{i-1}\colon V_t \in \{0,u,2u\} \setminus V_{t_{i-1}}\}.$$
Let $\rho_i = t_i-t_{i-1}$ for $i=1,\dots,N = \min\{i\colon V_{t_i}=0 \ \text{or} \ t_i=\tau_b(j)\} = \min\{i \colon t_i = \tau_c(j) \wedge \tau_b(j)\}$. By the a priori bound, we may assume $|V_{t_0}| \le nx_\star/2 = n\dlt/2r$. Then, $\xi_t= e^{\dlt (t_0 + t \wedge \rho_1)/2}V_{t_0 + t \wedge \rho_1}$ is a supermartingale with $\xi_0 \le n\dlt/2r$, so
$$\P(\rho_1>t) \le \P(\xi_t > e^{\dlt t/2}u) \le e^{-2\dlt t}\frac{n\dlt/2r}{\sqrt{n}-1} \sim e^{-2\dlt t}c/2r.$$
The above probability is $O(1/c)=o(1)$ if $t=(1/\dlt)\log c$. Using a similar estimate with $\xi_0=2u$,
$$\P(\rho_i>t \mid V_{\rho_{i-1}}=2u) \le e^{-2\dlt t}\frac{2u}{u} \le 2e^{-2\dlt t},$$
and integrating over $t$, $\E[\rho_i \mid V_{\rho_{i-1}}=2u] \le 1/\dlt$. To estimate $\rho_2$, we note that for $\alpha>0$
\begin{align*}
\mu_t(\alpha V^2) &= 2 \alpha V_t\mu(V_t) + \alpha^2\sigma^2_t(V) \\
& \ge -3\alpha \dlt V_t^2 + \alpha^2n\dlt/r
\end{align*}
so for $t_1\le t < t_2$, since $V_t \le 2u \le 2\sqrt{n}$, choosing $\alpha=13r$ we have $\mu_t(\alpha V^2) \ge  \alpha n\dlt$. Thus $V_{t_1 + t \wedge \rho_2}^2-n\dlt(t \wedge \rho_2)$ is a submartingale. Since $V_{t_1}=u$ and $V_{t_2} \le 2u$, using optional stopping,
$$\E[\rho_2] \le \frac{1}{n\dlt}(\E[V_{t_2}^2-V_{t_1}^2]) \le \frac{1}{n\dlt}(4u^2-u^2) \le \frac{3n}{n\dlt} = \frac{3}{\dlt}.$$
By the Markov property, the same estimate holds for $\E[\rho_i \mid V_{t_{i-1}}=u]$. Using simply that $\mu_t(V) \le 0$ and optional stopping, $\P(V_{t_i}=2u \mid V_{t_{i-1}}=u) \le 1/2$. Summarizing, on the time interval $[\tau_\star(j)\vee \tau_\star(n),\tau_c(j) \wedge \tau_b(j)]$, $V$ hits $u$, then goes to $2u$ and back to $u$ at most geometric$(1/2)$ number of times before either $V_{t_i}=0$ or $t_i=\tau_b(j)$. As shown above, $\rho_1$, the time to first hit $u$, is at most $(1/\dlt)\log c$ with probability $1-o(1)$, and the expected time to go from $u$ to either $0$, or to $2u$ and back to $u$ is at most $(4/\dlt)$. Using Wald's lemma and $\E[\text{geometric}(1/2)=2]$,
$$\sum_{i=2}^{N-1} \rho_i \le 8/\dlt.$$
Using Markov's inequality, the sum is at most $(1/\dlt)\log c$ with probability $1-o(1)$, so combining with the estimate on $\rho_1$, we find that uniformly over $j \in \{1,\dots,n\}$,
$$\P(\tau_b(j) \wedge \tau_c(j) - \tau_\star(j) \vee \tau_\star(n) > (2/\dlt)\log c \mid \Phi(j,\tau_\star(j))=X_\star)=o(1).$$
From Lemma \ref{lem:equil-time}, if $\dlt \to 0$ then $\E[\tau_\star(j)],\E[\tau_\star(n)]=O(c^2\log(c)/\dlt)$ uniformly over $j\in \{1,\dots,n\}$, so using Markov's inequality, with probability $1-o(1)$ uniformly in $j$, $\tau_\star(j)\vee \tau_\star(n) = o(e^{\eps c^2}/\dlt)$ for any fixed $\eps>0$. Summing the two and combining with \eqref{eq:apriori}, we obtain \eqref{eq:coup-time}.
\end{proof}

\section*{Appendix: Stochastic Calculus}
We recall a useful probability estimate and diffusion limit result, stated in the context of semimartingales. We give here a very brief list of definitions, enough for the acquainted reader to understand the context for this paper -- for an overview of the theory see \cite{jacod}.
Recall that a semimartingale (abbreviate s-m) is an optional process $X$ that can be written 
$$X=X_0 + M + A,$$
where $M$ is a local martingale and $A$ has finite variation. It is special if $A$ can be taken to be predictable, in which case we write
$$X = X_0+X^m + X^p,$$
where $X^m$ is the martingale part and $X^p$ is the (predictable) compensator. A sufficient condition for $X$ to be special is if it has bounded jumps, i.e., if the process of jumps $\Delta X_t = X_t - X_{t^-}$ satisfies $|\Delta X| \le \gamma$ for some non-random $\gamma<\infty$. If so, let $\Delta_\infty(X)$ denote the least such $\gamma$. In this case, a fortiori $X^m$ is locally square-integrable, i.e., the predictable quadratic variation $\langle X \rangle$ exists.\\

A process is \emph{quasi-left continuous} (qlc) if $\Delta X_T=0$ a.s.~on $\{T<\infty\}$ for any predictable time $T$. Feller processes, which include continuous time Markov chains, are quasi-left continuous. As noted in \cite{naming-game}, if $X$ is special and $X^m$ is locally square-integrable then $X$ is quasi-left continuous iff both $\langle X^m \rangle$ and $X^p$ are continuous. This motivates the following definition (not found in other references):
\begin{definition}
Let $X$ be a special s-m with $X^m$ locally square-integrable. Then $X$ is \emph{quasi-absolutely continuous} or qac if both $X^p$ and $\langle X^m \rangle$ are absolutely continuous. In this case define the \emph{drift} $\mu(X)$ and \emph{diffusivity} $\sigma^2(X)$ by
\begin{equation}\label{eq:mu-sig}
\mu_t(X) = \frac{d}{dt}X^p_t, \quad \sigma^2_t(X) = \frac{d}{dt}\langle X^m \rangle_t.
\end{equation}
\end{definition}

Any right-continuous continuous-time Markov chain $X$ on a finite state space $S \subset \R$ has finite variation so is a s-m. 
Index the possible transitions by $i \in \{1,\dots,m\}$ for some $m$, with $q_i:S\to \R_+$ the rates and $\Delta_i:S\to S-S$ the jumps. Writing $X$ as a sum of jumps and using the standard linear and quadratic martingales for Poisson processes, it is easy to show that $X$ is qac and has
\begin{align}\label{eq:mc-dd}
\mu_t(X) = \sum_{i=1}^m q_i(X_t)\Delta_i(X_t) \quad \text{and} \quad \sigma^2_t(X) = \sum_{i=1}^m q_i(X_t)(\Delta_i(X_t))^2.
\end{align}

Our first result gives a strong (exponential in $\mu/\sigma^2$) lower bound on the escape time from a barrier with negative drift. It is proved in \cite{SIS-SDE}.
\begin{lemma}[Drift barrier]\label{lem:driftbar}
 Fix $x>0$ and let $X$ be a qac s-m on $\R$ with jump size $\Delta_\infty(X) \le x/2$. Suppose there are positive reals $\mu_\star,\sigma^2_\star,C_{\mu_\star},C_\Delta$ with $\max\{\Delta_\infty(X)\mu_\star/\sigma^2_\star, 1/2\} \le C_\Delta$ so that if $0<X_t<x$ then
 $$ \mu_t(X) \leq -\mu_\star,\quad |\mu_t(X)| \leq C_{\mu_\star} \quad\hbox{and}\quad \sigma^2_t(X) \leq \sigma^2_\star.$$
Let $\Gamma = \exp(\mu_\star x /(32C_\Delta\sigma^2_\star))$. Then we have
 \begin{equation}\label{eq:driftbar}
 P\left( \ \sup_{t \le \lfloor \Gamma \rfloor x/16C_{\mu_\star} }X_t \ge x \ \mid \ X_0 \leq x/2 \ \right) \le 4/\Gamma.
 \end{equation}
\end{lemma}

The next result gives a diffusion limit, assuming the drift and diffusivity converge while the jump size tends to 0. It follows from Theorem 4.1 in \cite[Chapter 7]{EthierKurtz}, and from the Lipschitz existence and uniqueness condition for SDEs, if (i) in the proof of the former result we let $\tau_n^R$ be the exit time of $X^n$ from $(1/R,R)$ instead of $(-R,R)$, as described below, and (ii) we allow that the limiting diffusion $Y$ may be defined only on the interval $[0,\zeta)$ where $\zeta:=\lim_{R\to\infty} \tau^R$, with $\tau^R$ as defined below.

\begin{lemma}[Diffusion limit]\label{lem:limproc}
Let $X^n$ be a sequence of qac semimartingales with drift and diffusivity given by functions $\mu_n,\sigma^2_n$, and suppose $a:(0,\infty)\to \R_+$ and $b:(0,\infty)\to \R$ are such that $\sqrt{a}$ and $b$ are Lipschitz on compact subsets of $(0,\infty)$. Suppose the largest jump in $X^n$ tends to $0$ as $n\to\infty$. Also assume that for each $R>0$, as $n\to\infty$
$$\sup_{|x| \le R}|\mu_n(x) - b(x)|,|\sigma_n^2(x) - a(x)| \to 0.$$
Suppose $X^n(0) \to x \in \R$ and let $\tau_n^R = \inf\{t\colon X^n(t) \notin (1/R,R) \ \text{or} \ X^n(t^-) \notin (1/R,R)\}$. Then for all but countably many $R$, $X^n(\cdot \wedge \tau_n^R)$ converges in distribution to $X(\cdot\wedge \tau^R)$, where $X$ solves the initial value problem
$$x_0=x \quad \text{and}\quad dx = b(x)dt + \sqrt{a(x)}dB$$
and $\tau^R = \inf\{t\colon X(t) \notin (1/R,R)\}$.
\end{lemma}


\begin{thebibliography}{99}


\bibitem{thresh-lim}
Andersson, H. and Djehiche, B. (1998)
A threshold limit theorem for the stochastic logistic epidemic.
{\it J Appl Prob.} 35(3), 662--670

\bibitem{athreyabp}
Athreya, K.B. and Ney, P.E. (1972)
{\it Branching Processes}.
Springer-Verlag.

\bibitem{rand-init}
Barbour, A.D., Chigansky, P. and Klebaner, F. (2015)
On the emergence of random initial conditions in fluid limits.
{\it J Appl Prob.} 53(4),
DOI 10.1017/jpr.2016.74.

\bibitem{SIS-SDE}
Basak, A., Durrett, R. and Foxall, E. (2018)
Diffusion limit for the partner model at the critical value.
{\it Electron. J. Probab.} Vol 23, paper no. 102, 42 pp.


\bibitem{bright-luz}
Brightwell, G., House, T., and Luczak, M. (2018)
Extinction times in the subcritical stochastic SIS logistic epidemic.
{\it J Math Biol.} https://doi.org/10.1007/s00285-018-1210-5


\bibitem{doering2005}
Doering, C.R., Sargsyan, K.V. and Sander, L.M. (2005)
Extinction times for birth-death processes: exact results,
continuum asymptotics, and the failure of the Fokker–Planck approximation.
{\it Multiscale Model Simul} 3, 283--299

\bibitem{crit-scale}
Dolgoarshinnykh, R.G. and Lalley, S.P. (2006)
Critical scaling for the SIS stochastic epidemic.
{\it J Appl Probab} 43(3), 892--898.


\bibitem{PTE}
Durrett, R. (2010)
{\it Probability: Theory and Examples.}
Fourth Edition, Cambridge U. Press.

\bibitem{EthierKurtz}
Ethier, S.N., and Kurtz, T.G. (1986)
{\it Markov Processes: Characterization and Convergence.}
John Wiley and Sons, New York. 


\bibitem{feller-1939german}
Feller, W. (1939)
Die Grundlagen der Volterraschen Theorie des Kampfes ums Dasein in wahrscheinlichkeitstheoretischer Behandlung.
{\it Acta Biotheoretica.} 5(1), 11--40.


\bibitem{feller-works1}
Feller, W. (2015)
{\it Selected Papers I.}
Springer

\bibitem{naming-game}
Foxall, E. (2018)
The naming game on the complete graph. 
{\it Electron J. Probab.}.
Vol 23, paper no 126.

\bibitem{jacod}
Jacod, J. and Shiryaev, A.N. (2002).
{\it Limit theorems for stochastic processes.}
Springer Science and Business Media.

\bibitem{kallenberg}
Kallenberg, O. (1997)
{\it Foundations of modern probability.}
Springer Science and Business Media.

\bibitem{rar-exp}
Keilson, J. (1979)
{\it Markov Chain Models -- Rarity and Exponentiality.}
Applied Mathematical Sciences, vol. 28. Springer.

\bibitem{kryscio}
Kryscio, R.J. and Lef{\'e}vre, C. (1989)
On the extinction of the S-I-S stochastic logistic epidemic.
{\it J Appl Probab} 26(4), 685--694.

\bibitem{Kurtz-DDMC}
Kurtz, T.G. (1978)
Strong approximation theorems for density-dependent Markov chains.
{\it Stochastic Processes and their Applications.} 6(3), 223--240.

\bibitem{nasell1996}
Nasell, I. (1996)
The quasi-stationary distribution of the closed endemic SIS model.
{\it Adv. Appl. Prob.} 28, 895--932.

\bibitem{nasell1999}
Nasell, I. (1999)
On the quasi-stationary distribution of the stochastic logistic epidemic.
{\it Math Biosci} 156(1-2), 21--40

\bibitem{nasellmono}
Nasell, I. (2011)
{\it Extinction and Quasi-Stationarity in the Stochastic Logistic SIS Model.}
Springer.

\bibitem{protter}
Protter, P. (2005)
{\it Stochastic Integration and Differential Equations}, Second Edition.
Springer-Verlag.

\bibitem{trap-rule}
Talvila, E. and Wiersma, M. (2012)
Simple derivation of basic quadrature formulas.
{\it ArXiv:1202.0249}

\end{thebibliography}
\end{document}